\documentclass[reqno,A4paper]{amsart}

\usepackage{amsmath}
\usepackage{amssymb}
\usepackage{amsthm}
\usepackage{enumerate}
\usepackage{mathrsfs} 
\usepackage{eqlist}
\usepackage{array}

\usepackage[english]{babel}                             
\usepackage[latin1]{inputenc}
\usepackage{comment}
\usepackage{ae}
\usepackage[leqno]{amsmath}  

\usepackage[dotinlabels]{titletoc}     

\usepackage{amssymb,latexsym,verbatim} 
\usepackage{amstext}
\usepackage{amsfonts}
\usepackage{amsbsy}
\usepackage{amsopn}
\usepackage{enumerate}
\usepackage{txfonts}
\usepackage{amsmath, amsthm, amssymb}

\usepackage{color}
\usepackage[usenames,dvipsnames,svgnames,table]{xcolor}
\usepackage[numbers,sort&compress]{natbib}
\allowdisplaybreaks
\newtheorem{theorem}{Theorem}
\newtheorem{definition}[theorem]{Definition}
\newtheorem{lemma}[theorem]{Lemma}
\newtheorem{corollary}[theorem]{Corollary}
\newtheorem{proposition}[theorem]{Proposition}
\newtheorem{remark}[theorem]{Remark}
\numberwithin{equation}{section}
\numberwithin{theorem}{section}

\renewcommand{\epsilon}{\varepsilon}
\renewcommand{\rho}{\varrho}

\DeclareMathOperator*{\esssup}{ess\,sup}

\def\Xint#1{\mathchoice
{\XXint\displaystyle\textstyle{#1}}%
{\XXint\textstyle\scriptstyle{#1}}%
{\XXint\scriptstyle\scriptscriptstyle{#1}}%
{\XXint\scriptscriptstyle\scriptscriptstyle{#1}}%
\!\int}
\def\XXint#1#2#3{{\setbox0=\hbox{$#1{#2#3}{\int}$ }
\vcenter{\hbox{$#2#3$ }}\kern-.6\wd0}}

\def\dashint{\Xint-}

\usepackage{graphicx} 
\def\Yint#1{\mathchoice
    {\YYint\displaystyle\textstyle{#1}}%
    {\YYint\textstyle\scriptstyle{#1}}%
    {\YYint\scriptstyle\scriptscriptstyle{#1}}%
    {\YYint\scriptscriptstyle\scriptscriptstyle{#1}}%
      \!\iint}
\def\YYint#1#2#3{{\setbox0=\hbox{$#1{#2#3}{\iint}$}
    \vcenter{\hbox{$#2#3$}}\kern-.51\wd0}}
\def\longdash{{-}\mkern-3.5mu{-}}

\def\fiint{\Yint\longdash}

\numberwithin{equation}{section}

\begin{document}

\title[Higher integrability for obstacle problem ]%
{Higher integrability for obstacle problem related to the singular porous medium equation
}
\author[Qifan Li]%
{Qifan Li*}

\newcommand{\acr}{\newline\indent}

\address{\llap{*\,}Department of Mathematics\acr
                   School of Sciences\acr
                   Wuhan University of Technology\acr
                   430070, 122 Luoshi Road,
                   Wuhan, Hubei\acr
                   P. R. China}
\email{qifan\_li@yahoo.com, qifan\_li@whut.edu.cn}


\begin{abstract}
This paper is concerned with the self-improving property for obstacle problem related to the singular porous medium equation.
We establish a local higher integrability result for the spatial gradient of the $m$-th power of nonnegative weak solutions,
under some suitable regularity assumptions on the obstacle function.
\end{abstract}
\maketitle
\section{Introduction}
We are concerned in this paper with the
self-improving property for the gradient of nonnegative weak solutions to the obstacle problems related to the porous medium equation.
The porous medium equation
\begin{equation*}
\partial_t u-\Delta u^m=0,\qquad m>0,\end{equation*}
is an important prototype of nonlinear diffusion equation.
This kind of equation can be derived from modelling the flow of isentropic gas through a porous medium, models for groundwater infiltration
or heat radiation in plasmas (see for instance \cite[Chapter 2]{V}).

B\"ogelein, Lukkari and Scheven \cite{BLS} introduced the concept of obstacle problem related to the porous medium equation.
This kind of obstacle problem is a variational inequality
subject to a constraint that the solution should lie above a given obstacle function.
In \cite{BLS} the authors established the existence and uniqueness results for the strong and weak solutions to the obstacle problem.
Subsequently, the same authors \cite{BLS2017} obtained
a local H\"older continuity result of nonnegative weak solutions in the degenerate case
$m\geq 1$. In the fast diffusion range $\frac{(n-2)_+}{n+2}<m<1$, Cho and Scheven \cite{CS18} established the local H\"older continuity result for the
nonnegative weak solutions. Recently, Cho and Scheven \cite{CS19} proved the higher integrability of signed weak solutions to
the obstacle problems
in the degenerate range $m\geq 1$. Motivated by this work, we will study the higher integrability of nonnegative weak solutions to the obstacle problems
in the fast diffusion range $\frac{(n-2)_+}{n+2}<m<1$. This problem is at present far from being solved.

The higher integrability for the solutions of parabolic systems was first studied by Kinnunen and Lewis \cite{KL1,KL2}.
The treatment of the porous medium type equations is much more difficult. The higher integrability result for porous medium equations has been established by Gianazza and Schwarzacher \cite{GS2019,GS}; see also \cite{BDKS,BDKS2} for the case of porous medium systems.
For the treatment of obstacle problem related to the singular porous medium equation, our proof closely follows the scheme of \cite{GS}. We shall work with the sub-intrinsic cylinders constructed in
\cite{GS}.
In order to obtain gradient estimates on intrinsic cylinders,
we will distinguish between the degenerate case and the non-degenerate case.
Combining energy estimates, gluing lemma and the parabolic Sobolev inequality, we
establish a reverse H\"older inequality for the gradient of the $m$-th power of solutions on each intrinsic cylinder. The main difficulty in our proof is the treatment of the obstacle function. In order to obtain a suitable $L^\infty$ bound for the solution, we have to impose a condition
that $\psi^m$ is locally Lipschitz continuous, where $\psi$ is an obstacle function.
We also assume that the time derivative $\partial_t\psi^{1-m}$, that appears in the gluing lemma, is locally bounded.
Furthermore, we use a certain stopping time argument for the covering of the superlevel set of the gradient.
Contrary to the argument in \cite[section 7]{GS}, we use a localized maximal function instead of the strong maximal function, since the
localized version can be adapted to address obstacle problems.

The present paper is built up as follows. In \S 2, we set up notations and state the main result. \S 3 presents some preliminaries and
we explain the construction of the
sub-intrinsic cylinders. In \S 4, we establish the energy estimates, while in \S 5 we prove a gluing lemma which
describes the difference of two spatial averages. In \S 6, we establish the intrinsic reverse H\"older
inequalities for the gradient on intrinsic cylinders.
Finally the proof of the main result is presented in \S 7.
\section{Statement of the main result}
In the present section, we introduce the notations and give the statement of the main result.
Throughout the paper, we assume that $\Omega$ is a bounded domain in $\mathbb{R}^n$ with $n\geq 2$. For $T>0$,
let $\Omega_T$ denote the space-time cylinder $\Omega\times(0,T)$. Given a point $z_0=(x_0,t_0)\in \mathbb{R}^{n+1}$
and two parameters $r$, $s>0$, we set $B_r(x_0)=\{x\in\mathbb{R}^n:\ |x-x_0|<r\}$, $\Lambda_s(t_0)=(t_0-s,t_0+s)$
and $Q_{r,s}(z_0)=B_r(x_0)\times \Lambda_s(t_0)$. If the reference point $z_0$ is the origin, then we simply write
$B_r$, $\Lambda_s$ and $Q_{r,s}$ for $B_r(0)$, $\Lambda_s(0)$ and $Q_{r,s}(0)$. In this work we study
obstacle problems related to the quasilinear parabolic equations of the form
\begin{equation}\label{parabolic}\partial_t u-\operatorname{div}A(x,t,u,Du^m)=0.\end{equation}
Here, the vector field $A$ is only assumed to be measurable and satisfies
 \begin{equation}\label{A}
	\begin{cases}
	A(x,t,u,\zeta)\cdot\zeta\geq \nu_0|\zeta|^2,\\
	|A(x,t,u,\zeta)|\leq \nu_1|\zeta|,
	\end{cases}
\end{equation}
where $\nu_0$ and $\nu_1$ are fixed positive constants. Throughout the work, we only consider the singular case
$m\in \left(\frac{(n-2)_+}{n+2},1\right)$.
The obstacle problem for the porous medium type equation \eqref{parabolic}-\eqref{A} can be formulated as follows. Given
an obstacle function $\psi:\Omega_T\rightarrow\mathbb{R}_+$
with $D\psi^m\in L^2(\Omega_T)$ and $\partial_t\psi^m\in
L^{\frac{m+1}{m}}(\Omega_T)$,
we define the function classes
\begin{equation*}K_\psi=\left\{v\in C^0([0,T];L^{m+1}(\Omega)):v^m\in L^2(0,T;H
^1(\Omega)),v\geq\psi\ a.e.\ \text{in}\ \Omega_T\right\}\end{equation*}
and $K_\psi^\prime =\left\{v\in K_\psi:\partial_tv^m\in L^{\frac{m+1}{m}}(\Omega_T)\right\}$.
Let $\alpha\in W_0^{1,\infty}([0,T],\mathbb{R}_+)$ be a cut-off function in time and $\eta\in W_0^{1,\infty}(\Omega,\mathbb{R}_+)$ be a
cut-off function in space.
We define
\begin{equation*} 	\langle \!\langle\partial_t u,\alpha\eta (v^m-u^m)\rangle \!\rangle=
\iint_{\Omega_T}\eta\left[\alpha^\prime\left(\frac{1}{m+1}u^{m+1}-uv^m\right)-\alpha u\partial_tv^m\right]\,\mathrm {d}x\mathrm{d}t.\end{equation*}
The definition of
weak solutions to the obstacle problems related to the porous medium equation was first introduced by B\"ogelein, Lukkari and Scheven \cite{BLS}.
Cho and Scheven \cite{CS18} later extended the definition to the general quasilinear structure. In this paper, we
adopt the definition from \cite{CS18}.
\begin{definition}\label{definition}\cite{CS18}
A nonnegative function $u\in K_\psi$ is a local weak solution to the obstacle problem related to the porous medium type equation \eqref{parabolic}-\eqref{A} if
and only if the variational
inequality
\begin{equation}\label{weak identity} 	\langle \!\langle\partial_t u,\alpha\eta (v^m-u^m)\rangle \!\rangle+
\iint_{\Omega_T}\alpha A(x,t,u,Du^m)\cdot D(\eta(v^m-u^m))\,\mathrm {d}x\mathrm{d}t\geq0\end{equation}
holds true for any $v\in K_\psi^\prime$, any cut-off function in time $\alpha\in W_0^{1,\infty}([0,T],\mathbb{R}_+)$ and any cut-off function in space $\eta\in W_0^{1,\infty}(\Omega,\mathbb{R}_+)$.
\end{definition}
In this work, we shall make two regularity assumptions on the obstacle function under consideration.
More precisely, we assume that the obstacle function $\psi$ satisfies the following regularity properties:
\begin{itemize}
\item[(1)] The function $\psi^m$ is locally Lipschitz continuous in $\Omega_T$,
\item[(2)]  The time derivative $\partial_t\psi^{1-m}$ is locally bounded in $\Omega_T$.
\end{itemize}
The first assumption will be needed for the proof of Lemma \ref{sup} in \S 6, and the second assumption will be used to simplify estimating
the weighted spatial averages from \S 5. We emphasize that the second assumption can be improved to an integrability condition, but the proof is too long to give here.

According to \cite{CS18}, the assumption (1) implies that the weak solution $u$ is locally bounded and H\"older continuous in $\Omega_T$.
There is no loss of generality in assuming
\begin{equation}\label{01}0\leq u(x,t)\leq1\end{equation}
for all $(x,t)\in\Omega_T$.
For simplicity of notation, we write $\Psi=\psi^{m+1}+|\partial_t\psi^m|^{\frac{m+1}{m}}+|D\psi^m|^2$.
We are now in a position to state our main theorem.
\begin{theorem}\label{main theorem}
Let $\mathfrak z_0\in\Omega_T$ be a fixed point, and
let $R<1$ be a fixed positive number
such that $Q_{8R,64R^2}(\mathfrak z_0)\subset\Omega_T$. Assume that there exists a constant $M_0>0$ such that
\begin{equation}\begin{split}\label{lipschitz}
\sup_{Q_{8R,64R^2}(\mathfrak z_0)}(\ \Psi^{\frac{1}{m+1}}+|\partial_t\psi^{1-m}|^{\frac{1}{1-m}}\ )\leq M_0.
\end{split}\end{equation}
Let $u$ be a nonnegative weak solution to the obstacle problem in the sense of Definition \ref{definition} that satisfies \eqref{01}.
Then there exists a constant $\epsilon=\epsilon(n,m,\nu_0,\nu_1)>0$ such that
\begin{equation}\begin{split}\label{main result estimate}
\fiint_{Q_{R,R^2}(\mathfrak z_0)}&|Du^m|^{2(1+\epsilon)}\,\mathrm{d}x\mathrm{d}t
\leq
\gamma \left(\ \fiint_{Q_{4R,16R^2}(\mathfrak z_0)}|Du^m|^2\,\mathrm{d}x\mathrm{d}t\right)^{1+\epsilon}
+\gamma \left(
M_0^2+R^{-2}+1\right)^{1+\epsilon},
\end{split}\end{equation}
where the constant $\gamma$ depends only upon $n$, $m$, $\nu_0$ and $\nu_1$.
\end{theorem}
\begin{remark}
Contrary to \cite[Theorem 7.4]{GS}, which established a Calder\'on-Zygmund type estimate for the porous medium equation, we only derive the reverse H\"older inequality for the obstacle problem. Our proof makes no appeal to address the Calder\'on-Zygmund type estimate.
Finally, for the proof of Theorem \ref{main theorem}, we will write $\mathfrak z_0=(0,0)$ for simplicity of presentation.
\end{remark}

\section{Preliminary material}
In this section, we provide
some preliminary lemmas.
All the materials in this section are stated without proof. We first note that the weak solution to the obstacle problem
may not be differentiable in the time variable. In order to handle the problem with the time derivative, we will use the
following time mollification. For a fixed $h>0$, we set
  \begin{equation*}[\![v]\!]_h(x,t)=\frac{1}{h}\int_0^te^{\frac{s-t}{h}}v(x,s)\,\mathrm{d}s,\end{equation*}
where $v\in L^1(\Omega_T)$. Some basic properties of the time mollification are listed in the following lemma (see for instance
\cite[Lemma 3.1]{BLS}).
\begin{lemma}  \label{time mollification}
Let $p\geq1$ and assume that $v\in L^1(\Omega_T)$.
\begin{itemize}
\item[(1)] For a fixed $h>0$, there holds
$\partial_t[\![v]\!]_h=h^{-1}(v-[\![v]\!]_h).$
\item[(2)] If $v\in L^p(0,T;W^{1,p}(\Omega))$, then $[\![v]\!]_h\to v$ and $D[\![v]\!]_h\to Dv$ in $L^p(\Omega_T)$ as $h\downarrow0$.
\item[(3)] If $v\in C(\Omega_T)$, then $[\![v]\!]_h\to v$ uniformly in $\Omega_T$ as $h\downarrow0$.
\end{itemize}
\end{lemma}
We remark that Lemma \ref{time mollification} (3) applies to the weak solution $u$, since the weak solution to the obstacle problem is locally H\"older continuous. Next, we recall the inequalities which was obtained from \cite[Proposition 2.1]{GS}.
\begin{lemma}\label{auxiliary}\cite[Proposition 2.1]{GS} Suppose that $u$, $c\geq0$ and $0<m<1$, we have
\begin{equation}\label{auxiliary1}\frac{1}{2}(u-c)(u^m-c^m)\leq \int_c^u(y^m-c^m)\,\mathrm{d}y\leq (u-c)(u^m-c^m)\qquad\text{if}\quad u\geq c,\end{equation}
\begin{equation}\label{auxiliary2}\frac{m}{2}(c-u)(c^m-u^m)\leq \int_u^c(c^m-y^m)\,\mathrm{d}y\leq (c-u)(c^m-u^m)\qquad\text{if}\quad u< c.\end{equation}
\end{lemma}
We note that Lemma \ref{auxiliary} will be used to derive the energy estimates
in \S 4. This lemma also plays a crucial role in the proof of Lemma \ref{gluing lemma} in \S 5.
Furthermore, we recall the definitions of intrinsic and sub-intrinsic cylinders which was introduced from \cite[section 3]{GS}.
\begin{definition}\cite{GS} Let $z_0\in \Omega_T$ be a fixed point, and let $r$, $\theta>0$ such that $Q_{r,\theta r^2}(z_0)
\subset \Omega_T$. We say that $Q_{r,\theta r^2}(z_0)$ is a sub-intrinsic cylinder if and only if the following inequality
holds:
$$\fiint_{Q_{r,\theta r^2}(z_0)}u^{m+1}\,\mathrm{d}x\mathrm{d}t\leq K_1\theta^{\frac{m+1}{1-m}},$$
where the constant $K_1\geq 1$. Moreover, we say that $Q_{r,\theta r^2}(z_0)$ is an intrinsic cylinder if and only if
$$K_2^{-1}\theta^{\frac{m+1}{1-m}}\leq  \fiint_{Q_{r,\theta r^2}(z_0)}u^{m+1}\,\mathrm{d}x\mathrm{d}t\leq K_2\theta^{\frac{m+1}{1-m}}$$
holds for some constant $K_2\geq 1$.
\end{definition}
At this point, we follow the idea in \cite{GS} to construct the sub-intrinsic cylinders which will be used in the covering argument in \S 7.
Let $z_0=(x_0,t_0)\in \Omega_T$ be a point such that $Q_{R,R^2}(z_0)
\subset \Omega_T$. For any $s\in (0,R^2]$, we denote by $\tilde r(s)$ the quantity
\begin{equation}\label{tilde r}\tilde r(s)=\sup\left\{r<R:\ \left(\int_{t_0-s}^{t_0+s}
\int_{B_r(x_0)}u^{m+1}\,\mathrm{d}x\mathrm{d}t\right)^{1-m}r^{2(m+1)}|B_r|^{m-1}\leq s^2\right\}.\end{equation}
Let $b_0=(n+2)(m+1)-2n$ and let $\hat b\in (0,\min\{b_0,\frac{1}{2}\})$.
According to the proof of \cite[Lemma 3.1]{GS}, the function $\tilde r(s)$ is
continuous and this enables us to introduce the radius
\begin{equation}\label{r}r(s)=r(s,z_0)=\min_{s\leq t\leq R^2} \left(\frac{s}{t}\right)^{\hat{b}}\tilde r(t)\end{equation}
for any $s\in(0,R^2]$.
Subsequently, we write $ Q_s(z_0)=Q_{r(s),s}(z_0)$ and denote by $\theta_s(z_0)$ the quantity
$$\theta_s(z_0)=\frac{s}{r(s)^2}.$$
If $z_0=(0,0)$,  then we abbreviate $Q_s:= Q_s((0,0))$ and $\theta_s:=\theta_s((0,0))$. We now summarize the results obtained from
\cite{GS} for this kind of cylinder as follows.
\begin{lemma}\label{subcylinder}\cite{GS} Fix a point $z_0\in \Omega_T$ and assume that $Q_{R,R^2}(z_0)
\subset \Omega_T$. Let $s\in (0,R^2]$ and $r(s)$ be
the radius constructed via \eqref{tilde r}-\eqref{r}. Then, the
cylinder $Q_s(z_0)$ is sub-intrinsic and satisfies the following property:
\begin{itemize}
  \item [(1)] $\fiint_{Q_s(z_0)}u^{m+1}\,\mathrm{d}x\mathrm{d}t\leq \theta_s(z_0)^{\frac{m+1}{1-m}}$.
\end{itemize}
For $s$, $\sigma\in (0,R^2]$ and $s<\sigma$, we have the properties for the concentric cylinders
$Q_s(z_0)$ and $Q_\sigma(z_0)$ as follows:
\begin{itemize}
  \item [(2)] $r(s)\leq \left(\frac{s}{\sigma}\right)^{\hat {b}}r(\sigma)$ and $r(s)\to 0$ as $s\downarrow 0$.
  \item [(3)] If \ $\fiint_{Q_\tau(z_0)}u^{m+1}\,\mathrm{d}x\mathrm{d}t< \theta_\tau(z_0)^{\frac{m+1}{1-m}}$ holds for any $\tau\in (s,\sigma)$,
  then
  $$\theta_\tau(z_0)\leq \left(\frac{\tau}{\sigma}\right)^\beta\theta_\sigma(z_0),$$
  where $\beta=1-2\hat b>0$.
  \item [(4)] If $\sigma=ks$ for some $k\geq 1$, then there holds
  \begin{equation*}\begin{split}r(s)&\leq k^{-\hat b}r(ks)\leq k^{\hat a-\hat b}r(s),
  \\ \theta_{ks}(z_0)&\leq k^{\beta}\theta_{s}(z_0)\leq  k^{\beta+2\hat a}\theta_{ks}(z_0),
  \end{split}\end{equation*}
  where $\hat a=\hat b+\frac{2}{2(m+1)-(1-m)n}$.
  \item [(5)] If $\sigma=ks$ for some $k\geq 1$ and $Q_s(z_0)$ is intrinsic, then also the cylinder $Q_\sigma(z_0)$ is intrinsic.
\end{itemize}
Let $\mathfrak z_0\in\Omega_T$ be a point such that $Q_{8R,64R^2}(\mathfrak z_0)\subset\Omega_T$ and assume that for some $K\geq 1$,
$$\fiint_{Q_{8R,64R^2}(\mathfrak z_0)}u^{m+1}\,\mathrm{d}x\mathrm{d}t\leq K^{\frac{m+1}{1-m}}.$$
Then for any $z\in Q_{4R,16R^2}(\mathfrak z_0)$, we have the following global estimate
\begin{itemize}
  \item [(6)] $1\leq \theta_{R^2}(z)\leq cK^{2(m+1)\bar a},$
  where $\bar a=\frac{1}{2(m+1)+n(m-1)}$ and $c=c(n,m)$.
\end{itemize}
Furthermore, if $y,z\in Q_{4R,16R^2}(\mathfrak z_0)$ and $Q_{r(s,z),s}(z)\cap Q_{r(s,y),s}(y)\neq \emptyset$, then there exists a constant
$\hat c=\hat c(n,m,K)>1$ such that for any $0<s\leq \frac{R^2}{2\hat c}$ there holds
\begin{itemize}
  \item [(7)] $Q_{r(s,z),s}(z)\subset Q_{r(\hat cs,y),\hat cs}(y)$ and $Q_{r(s,y),s}(y)\subset Q_{r(\hat cs,z),\hat cs}(z)$.
\end{itemize}
\end{lemma}
In the applications, we can use the assumption \eqref{01} to deduce that
$$\fiint_{Q_{8R,64R^2}}u^{m+1}\,\mathrm{d}x\mathrm{d}t\leq 1.$$
This enables us to take $K=1$ when we apply Lemma \ref{subcylinder} (6) and (7).
As indicated in \cite{GS}, the properties (4) and (7) imply  the following Vitali-type covering property.
\begin{lemma}\label{vitali}\cite{GS} Let $V\subset Q_{4R,16R^2}(\mathfrak z_0)$ and let $Q_{r(s,z),s}(z)$ be the sub-intrinsic
cylinder as in Lemma \ref{subcylinder}. Let $\mathcal{F}=\{Q_{r(s,z),s}(z):\ z\in V\}$ be a covering of $V$. Then there exists
a countable family $\mathcal{G}=\{Q_{r(s_i,z_i),s_i}(z_i)\}_{i=1}^\infty$ of disjoint cylinders in $\mathcal{F}$ and a constant
$\chi=\chi(n,m)>1$ such that
\begin{equation*}\begin{split}V\subset \bigcup_{i=1}^\infty Q_{r(\chi s_i,z_i),\chi s_i}(z_i). \end{split}\end{equation*}
\end{lemma}
The Vitali-type covering Lemma will be used only in \S 7. This kind of covering plays a
crucial role in the proof of weak type estimate for the maximal function in Lemma \ref{maxlemma}. On the other hand, the Vitali-type covering
will be used to estimate the measure of superlevel sets of the gradient.
\section{Caccioppoli type inequalities}
The aim of this section is to establish energy estimates for the weak solution of the obstacle problem. Here, we state and prove the
energy estimates on the condition that the function $\Psi$ is locally integrable in $\Omega_T$. This condition is weaker than
the Lipschitz condition \eqref{lipschitz}. Our main result in this section states as follows.
\begin{lemma}\label{Caccioppoli}Let $0<m<1$ and let $u$ be a nonnegative weak solution to the obstacle problem in the sense of Definition \ref{definition}.
Let $z_0=(x_0,t_0)\in \Omega_T$ be a point such that
$Q_{r_1,s_1}(z_0)\subset Q_{r_2,s_2}(z_0)\subset \Omega_T$. Assume that $\phi\in C_0^\infty(B_{r_2}(x_0))$ and $0\leq \phi\leq1$.
Then there exists a constant $\gamma=\gamma(m,\nu_0,\nu_1)>0$ such that for any $c\geq0$ there holds
\begin{equation}\begin{split}\label{Cac1}\esssup_{t\in \Lambda_{s_1}(t_0)}&\int_{B_{r_2}(x_0)}\phi^2(u-c)_+(u^m-c^m)_+\,
\mathrm{d}x+\int_{\Lambda_{s_1}(t_0)}\int_{B_{r_2}(x_0)}|D[(u^m-c^m)_+\phi]|^2\,\mathrm{d}x\mathrm{d}t
\\&\leq \frac{\gamma}{s_2-s_1}\iint_{Q_{r_2,s_2}(z_0)}(u-c)_+(u^m-c^m)_+\,
\mathrm{d}x\mathrm{d}t\\&\quad+\gamma \iint_{Q_{r_2,s_2}(z_0)}(u^m-c^m)_+^2|\,D\phi\,|^2\,\mathrm{d}x\mathrm{d}t
\\&\qquad+\gamma\iint_{Q_{r_2,s_2}(z_0)}\left(\psi^{m+1}+|\partial_t\psi^m|^{\frac{m+1}{m}}+|D\psi^m|^2\right)\chi_{\{u>c\}}\,
\mathrm{d}x\mathrm{d}t.\end{split}\end{equation}
Moreover, for any $c\geq0$ we have
\begin{equation}\begin{split}\label{Cac2}\esssup_{t\in \Lambda_{s_1}(t_0)}&\int_{B_{r_1}(x_0)}|u-c|\,|u^m-c^m|\,
\mathrm{d}x+\iint_{Q_{r_1,s_1}(z_0)}|Du^m|^2\,\mathrm{d}x\mathrm{d}t
\\&\leq \frac{\gamma}{s_2-s_1}\iint_{Q_{r_2,s_2}(z_0)}|u-c|\,|u^m-c^m|\,
\mathrm{d}x\mathrm{d}t\\&\quad+\frac{\gamma}{(r_2-r_1)^2}\iint_{Q_{r_2,s_2}(z_0)}|u^m-c^m|^2\,\mathrm{d}x\mathrm{d}t
\\&\qquad+\gamma\iint_{Q_{r_2,s_2}(z_0)}\left(\psi^{m+1}+|\partial_t\psi^m|^{\frac{m+1}{m}}+|D\psi^m|^2\right)\,
\mathrm{d}x\mathrm{d}t,\end{split}\end{equation}
where the constant $\gamma$ depends only upon $\nu_0$, $\nu_1$ and $m$.
\end{lemma}
\begin{proof}We begin with the proof of \eqref{Cac1}, which is the most difficult part of the proof.
In the variational inequality \eqref{weak identity} we choose
\begin{equation}\label{comparison map}v^m=[\![u^m]\!]_h-([\![u^m]\!]_h-\psi_c^m)_++\|\psi^m-[\![\psi^m]\!]_h\|_{L^\infty(\Omega_T)}\end{equation}
as a comparison map,
where the function $\psi_c$ is defined by
$$\psi_c^m=\max\{c^m,\psi^m\}=c^m+(\psi^m-c^m)_+.$$
It is easy to check that $v\in K_\psi^\prime$.
We first remark that since $u\geq\psi$, two superlevel sets $\{u\geq c\}$ and $\{u\geq \psi_c\}$ are equal.
More precisely, the relation
\begin{equation}\label{set identity}\{x\in B_{r_2}(x_0):\ u(x,t)\geq c\}=\{x\in B_{r_2}(x_0):\ u(x,t)\geq\psi_c\}\end{equation}
holds true for any $t\in\Lambda_{s_2}(t_0)$. Let $\eta=\phi^2$ and $\alpha\in W_0^{1,\infty}([0,T],\mathbb{R}_+)$ be
a fixed cut-off function which will be determined later.

We now proceed to establish an energy estimate from the variational inequality \eqref{weak identity}.
For the first term on the left-hand side of \eqref{weak identity} we compute
\begin{equation}\begin{split}\label{time1}
\langle \!\langle&\partial_t u,\alpha\eta (v^m-u^m)\rangle \!\rangle\\&=
\iint_{\Omega_T}\eta\alpha^\prime\left(\frac{1}{m+1}u^{m+1}-u[\![u^m]\!]_h+u([\![u^m]\!]_h-\psi_c^m)_+-u
\|\psi^m-[\![\psi^m]\!]_h\|_{L^\infty(\Omega_T)}\right)\,\mathrm {d}x\mathrm{d}t
\\&\quad-\iint_{\Omega_T}\eta\alpha u\partial_tv^m\,\mathrm {d}x\mathrm{d}t.
\end{split}\end{equation}
In view of \eqref{comparison map}, we deduce
$$\partial_tv^m=\partial_t[\![u^m]\!]_h\chi_{\{[\![u^m]\!]_h\leq\psi_c^m\}}+\partial_t\psi_c^m\chi_{\{[\![u^m]\!]_h>\psi_c^m\}}$$
and the second term on the right-hand side of \eqref{time1} is estimated above by
\begin{equation*}\begin{split}
-&\iint_{\Omega_T}\eta\alpha u\partial_tv^m\,\mathrm {d}x\mathrm{d}t
\\&=-\iint_{\Omega_T\cap \{[\![u^m]\!]_h\leq\psi_c^m\}}\eta\alpha u\partial_t[\![u^m]\!]_h\,\mathrm {d}x\mathrm{d}t-\iint_{\Omega_T\cap\{[\![u^m]\!]_h>\psi_c^m\}}\eta\alpha u\partial_t\psi_c^m\,\mathrm {d}x\mathrm{d}t
\\&=-\iint_{\Omega_T\cap \{[\![u^m]\!]_h\leq\psi_c^m\}}\eta\alpha (u-[\![u^m]\!]_h^{\frac{1}{m}})\frac{1}{h}
\left(u^m-[\![u^m]\!]_h\right)\,\mathrm {d}x\mathrm{d}t
\\&\quad-\iint_{\Omega_T\cap \{[\![u^m]\!]_h\leq\psi_c^m\}}\eta\alpha [\![u^m]\!]_h^{\frac{1}{m}}
\partial_t[\![u^m]\!]_h\,\mathrm {d}x\mathrm{d}t
\\&\quad -\iint_{\Omega_T\cap\{[\![u^m]\!]_h>\psi_c^m\}}\eta\alpha u\partial_t\psi_c^m\,\mathrm {d}x\mathrm{d}t
\\&\leq-\iint_{\Omega_T\cap \{[\![u^m]\!]_h\leq\psi_c^m\}}\eta\alpha [\![u^m]\!]_h^{\frac{1}{m}}
\partial_t[\![u^m]\!]_h\,\mathrm {d}x\mathrm{d}t
 -\iint_{\Omega_T\cap\{[\![u^m]\!]_h>\psi_c^m\}}\eta\alpha u\partial_t\psi_c^m\,\mathrm {d}x\mathrm{d}t,
\end{split}\end{equation*}
where we have used the identity $\partial_t[\![u^m]\!]_h=h^{-1}(u^m-[\![u^m]\!]_h)$.
Noting that
\begin{equation*}\begin{split}\iint_{\Omega_T}&\eta\alpha [\![u^m]\!]_h^{\frac{1}{m}}
\left[\partial_t[\![u^m]\!]_h-\partial_t([\![u^m]\!]_h-\psi_c^m)_+\right]\,\mathrm {d}x\mathrm{d}t\\&=
\iint_{\Omega_T\cap \{[\![u^m]\!]_h\leq\psi_c^m\}}\eta\alpha [\![u^m]\!]_h^{\frac{1}{m}}
\partial_t[\![u^m]\!]_h\,\mathrm {d}x\mathrm{d}t+\iint_{\Omega_T\cap\{[\![u^m]\!]_h>\psi_c^m\}}\eta\alpha
[\![u^m]\!]_h^{\frac{1}{m}}\partial_t\psi_c^m\,\mathrm {d}x\mathrm{d}t,
\end{split}\end{equation*}
we have
\begin{equation*}\begin{split}
-&\iint_{\Omega_T}\eta\alpha u\partial_tv^m\,\mathrm {d}x\mathrm{d}t
\\&\leq -\iint_{\Omega_T}\eta\alpha [\![u^m]\!]_h^{\frac{1}{m}}
\left[\partial_t[\![u^m]\!]_h-\partial_t([\![u^m]\!]_h-\psi_c^m)_+\right]\,\mathrm {d}x\mathrm{d}t
\\&\quad+\iint_{\Omega_T\cap\{[\![u^m]\!]_h>\psi_c^m\}}\eta\alpha
[\![u^m]\!]_h^{\frac{1}{m}}\partial_t\psi_c^m\,\mathrm {d}x\mathrm{d}t
-\iint_{\Omega_T\cap\{[\![u^m]\!]_h>\psi_c^m\}}\eta\alpha u\partial_t\psi_c^m\,\mathrm {d}x\mathrm{d}t.
\end{split}\end{equation*}
Integrating by parts, we obtain
\begin{equation}\begin{split}\label{time2}
-&\iint_{\Omega_T}\eta\alpha u\partial_tv^m\,\mathrm {d}x\mathrm{d}t
\\&\leq\iint_{\Omega_T}\eta\alpha^\prime\frac{m}{m+1} [\![u^m]\!]_h^{\frac{m+1}{m}}
\,\mathrm {d}x\mathrm{d}t-\iint_{\Omega_T}\eta\alpha^\prime [\![u^m]\!]_h^{\frac{1}{m}}([\![u^m]\!]_h-\psi_c^m)_+\,\mathrm {d}x\mathrm{d}t
\\&\quad+\iint_{\Omega_T\cap\{[\![u^m]\!]_h>\psi_c^m\}}\eta\alpha
[\![u^m]\!]_h^{\frac{1}{m}}\partial_t\psi_c^m\,\mathrm {d}x\mathrm{d}t
-\iint_{\Omega_T\cap\{[\![u^m]\!]_h>\psi_c^m\}}\eta\alpha u\partial_t\psi_c^m\,\mathrm {d}x\mathrm{d}t
\\&\quad-\iint_{\Omega_T}\eta\alpha\partial_t \big([\![u^m]\!]_h^{\frac{1}{m}}\big)([\![u^m]\!]_h-\psi_c^m)_+\,\mathrm {d}x\mathrm{d}t.
\end{split}\end{equation}
Combining \eqref{time2} with \eqref{time1}, we infer that
\begin{equation}\begin{split}\label{time3}
\langle \!\langle&\partial_t u,\alpha\eta (v^m-u^m)\rangle \!\rangle\\&\leq
\iint_{\Omega_T}\eta\alpha^\prime\left(\frac{m}{m+1} [\![u^m]\!]_h^{\frac{m+1}{m}}
+\frac{1}{m+1}u^{m+1}-u[\![u^m]\!]_h\right)
\,\mathrm {d}x\mathrm{d}t
\\&\quad+\iint_{\Omega_T}\eta\alpha^\prime \big(u-[\![u^m]\!]_h^{\frac{1}{m}}\big)
([\![u^m]\!]_h-\psi_c^m)_+\,\mathrm {d}x\mathrm{d}t
\\&\quad-\iint_{\Omega_T}\eta\alpha^\prime \|\psi^m-[\![\psi^m]\!]_h\|_{L^\infty(\Omega_T)}\,\mathrm {d}x\mathrm{d}t
\\&\quad+\iint_{\Omega_T\cap\{[\![u^m]\!]_h>\psi_c^m\}}\eta\alpha
\big([\![u^m]\!]_h^{\frac{1}{m}}-u\big)\partial_t\psi_c^m\,\mathrm {d}x\mathrm{d}t
\\&\quad+\iint_{\Omega_T}(-1)\eta\alpha\partial_t \big([\![u^m]\!]_h^{\frac{1}{m}}\big)([\![u^m]\!]_h-\psi_c^m)_+\,\mathrm {d}x\mathrm{d}t
\\&=:\uppercase\expandafter{\romannumeral1}+\uppercase\expandafter{\romannumeral2}-\uppercase\expandafter{\romannumeral3}+
\uppercase\expandafter{\romannumeral4}+\uppercase\expandafter{\romannumeral5},
\end{split}\end{equation}
with the obvious meaning of $\uppercase\expandafter{\romannumeral1}$, $\uppercase\expandafter{\romannumeral2}$, $\uppercase\expandafter{\romannumeral3}$, $\uppercase\expandafter{\romannumeral4}$
and $\uppercase\expandafter{\romannumeral5}$. Observe that $[\![\psi^m]\!]_h\to \psi^m$ and
$[\![u^m]\!]_h\to u^m$ uniformly in $\Omega_T$ as $h\downarrow 0$,
since $\psi$ and $u$ are locally continuous. We apply Lebesgue's dominated convergence theorem to obtain
$\uppercase\expandafter{\romannumeral1}+\uppercase\expandafter{\romannumeral2}+\uppercase\expandafter{\romannumeral3}+
\uppercase\expandafter{\romannumeral4}\to 0$ as $h\downarrow 0$. It remains to treat the term $\uppercase\expandafter{\romannumeral5}$.

Noting that
\begin{equation*}\begin{split}
\frac{\partial}{\partial t}\big[\int_{\psi_c}^{[\![u^m]\!]_h^{\frac{1}{m}}}(y^m-\psi_c^m)_+\,\mathrm {d}y\big]=
\partial_t\big([\![u^m]\!]_h^{\frac{1}{m}}\big)([\![u^m]\!]_h-\psi_c^m)_+-
\partial_t \psi_c^m([\![u^m]\!]_h^{\frac{1}{m}}-\psi_c)_+,
\end{split}\end{equation*}
we use integration by parts to get
\begin{equation*}\begin{split}
\uppercase\expandafter{\romannumeral5}=&\iint_{\Omega_T}\eta\alpha^\prime\int_{\psi_c}^{[\![u^m]\!]_h^{\frac{1}{m}}}(y^m-\psi_c^m)_+\,\mathrm {d}y\,\mathrm {d}x\mathrm{d}t
-\iint_{\Omega_T}\eta\alpha\partial_t \psi_c^m([\![u^m]\!]_h^{\frac{1}{m}}-\psi_c)_+\,\mathrm {d}x\mathrm{d}t
\\=&:\uppercase\expandafter{\romannumeral5}_1+\uppercase\expandafter{\romannumeral5}_2,
\end{split}\end{equation*}
with the obvious meaning of $\uppercase\expandafter{\romannumeral5}_1$ and $\uppercase\expandafter{\romannumeral5}_2$.
We first observe that
\begin{equation*}\begin{split}
\uppercase\expandafter{\romannumeral5}_2\to
-\iint_{\Omega_T}\eta\alpha\partial_t \psi_c^m(u-\psi_c)_+\,\mathrm {d}x\mathrm{d}t
\qquad\text{as}\qquad  h\downarrow 0,
\end{split}\end{equation*}
since $[\![u^m]\!]_h\to u^m$ uniformly in $\Omega_T$ as $h\downarrow 0$.
Our next aim is to obtain lower and upper bounds for $\uppercase\expandafter{\romannumeral5}_1$. To this end,
we need to determine the cut-off function in time $\alpha(t)$.
For a fixed time level $t_1\in \Lambda_{s_1}(t_0)\subset(0,T)$, we define
\begin{equation}\label{alpha}
	\alpha(t)=\begin{cases}
0,&\quad \text{for}\quad t\in (0,t_0-s_2],\\
	1-\frac{1}{s_2-s_1}(t_0-s_1-t),&\quad \text{for}\quad t\in (t_0-s_2,t_0-s_1],
\\
	1,&\quad \text{for}\quad t\in (t_0-s_1,t_1-\epsilon],\\
	1-\frac{1}{\epsilon}(t-t_1+\epsilon),&\quad \text{for}\quad t\in (t_1-\epsilon,t_1],
\\ 0,&\quad \text{for}\quad t\in (t_1,T),
	\end{cases}
\end{equation}
where $0<\epsilon\ll 1$.
We now turn our
attention to the estimate of $\uppercase\expandafter{\romannumeral5}_1$. From \eqref{auxiliary1}, we find that
\begin{equation*}\begin{split}\uppercase\expandafter{\romannumeral5}_1&=
\iint_{\Omega_T\cap \{[\![u^m]\!]_h\geq \psi_c^m\}}\phi^2\alpha^\prime\int_{\psi_c}^{[\![u^m]\!]_h^{\frac{1}{m}}}(y^m-\psi_c^m)
\,\mathrm {d}y\,\mathrm {d}x\mathrm{d}t
\\&\leq \frac{1}{s_2-s_1}
\int_{t_0-s_2}^{t_0-s_1}\int_{B_{r_2}(x_0)\cap \{[\![u^m]\!]_h(\cdot,t)\geq \psi_c^m (\cdot,t)\}}([\![u^m]\!]_h^{\frac{1}{m}}-\psi_c)\,
([\![u^m]\!]_h-\psi_c^m)\,\mathrm {d}x\mathrm{d}t
\\&\quad-\frac{1}{2\epsilon}\int_{t_1-\epsilon}^{t_1}\int_{B_{r_2}(x_0)\cap \{[\![u^m]\!]_h(\cdot,t)\geq \psi_c^m (\cdot,t)\}}\phi^2([\![u^m]\!]_h^{\frac{1}{m}}-\psi_c)\,
([\![u^m]\!]_h-\psi_c^m)\,\mathrm {d}x\mathrm{d}t.
\end{split}\end{equation*}
Applying Lebesgue's dominated convergence theorem, we pass to the limit $ h\downarrow 0$ on the right-hand side and conclude that
\begin{equation*}\begin{split}
\lim\sup_{\epsilon\downarrow 0}\lim\sup_{h\downarrow 0}\uppercase\expandafter{\romannumeral5}_1&\leq \frac{1}{s_2-s_1}
\int_{t_0-s_2}^{t_0-s_1}\int_{B_{r_2}(x_0)\cap \{u(\cdot,t)\geq \psi_c (\cdot,t)\}}(u-\psi_c)\,
(u^m-\psi_c^m)\,\mathrm {d}x\mathrm{d}t
\\&\quad-\frac{1}{2}\int_{B_{r_2}(x_0)\cap \{u(\cdot,t_1)\geq \psi_c (\cdot,t_1)\}}\phi^2(\cdot)(u-\psi_c)(\cdot,t_1)\,
(u^m-\psi_c^m)(\cdot,t_1)\,\mathrm {d}x.
\end{split}\end{equation*}
From the preceding arguments, we infer from \eqref{time3} that for any $t_1\in \Lambda_{s_1}(t_0)$,
there holds
\begin{equation}\begin{split}\label{time4}
\lim\sup_{\epsilon\downarrow 0}\lim\sup_{h\downarrow 0}&\,\langle \!\langle\partial_t u,\alpha\eta (v^m-u^m)\rangle \!\rangle\\&\leq
 \frac{1}{s_2-s_1}
\int_{t_0-s_2}^{t_0-s_1}\int_{B_{r_2}(x_0)\cap \{u(\cdot,t)\geq \psi_c (\cdot,t)\}}(u-\psi_c)\,
(u^m-\psi_c^m)\,\mathrm {d}x\mathrm{d}t
\\&\quad+\iint_{Q_{r_2,s_2}(z_0)}|\partial_t \psi_c^m|\,(u-\psi_c)_+\,\mathrm {d}x\mathrm{d}t
\\&\quad-\frac{1}{2}\int_{B_{r_2}(x_0)\cap \{u(\cdot,t_1)\geq \psi_c (\cdot,t_1)\}}\phi^2(u-\psi_c)(\cdot,t_1)\,
(u^m-\psi_c^m)(\cdot,t_1)\,\mathrm {d}x
\\&=:\uppercase\expandafter{\romannumeral6}+\uppercase\expandafter{\romannumeral7}-\uppercase\expandafter{\romannumeral8},
\end{split}\end{equation}
with the obvious meaning of $\uppercase\expandafter{\romannumeral6}$, $\uppercase\expandafter{\romannumeral7}$ and $\uppercase\expandafter{\romannumeral8}$. To estimate $\uppercase\expandafter{\romannumeral6}$, we note that
$u-\psi_c\leq u-c$ on the set $\{u\geq \psi_c\}$. From this inequality and \eqref{set identity}, we conclude that
\begin{equation}\begin{split}\label{VI}
\uppercase\expandafter{\romannumeral6}&\leq \frac{1}{s_2-s_1}
\iint_{Q_{r_2,s_2}(z_0)\cap \{u\geq \psi_c\}}(u-c)\,
(u^m-c^m)\,\mathrm {d}x\mathrm{d}t
\\&=\frac{1}{s_2-s_1}
\iint_{Q_{r_2,s_2}(z_0)\cap \{u\geq c\}}(u-c)\,
(u^m-c^m)\,\mathrm {d}x\mathrm{d}t
\\&=\frac{1}{s_2-s_1}
\iint_{Q_{r_2,s_2}(z_0)}(u-c)_+\,
(u^m-c^m)_+\,\mathrm {d}x\mathrm{d}t.
\end{split}\end{equation}
We now come to the estimate of $\uppercase\expandafter{\romannumeral7}$. We first observe that $\{\psi>c\}\subset \{u>c\}$,
$\partial_t \psi_c=\partial_t(\psi^m-c^m)_+$ and
\begin{equation*}\begin{split}u=(u^m)^{\frac{1}{m}}=(u^m-c^m+c^m)^{\frac{1}{m}}\leq 2^{\frac{1-m}{m}}
((u^m-c^m)_+^{\frac{1}{m}}+c).\end{split}\end{equation*}
From this inequality, we conclude that
\begin{equation*}\begin{split}
\uppercase\expandafter{\romannumeral7}&\leq\iint_{Q_{r_2,s_2}(z_0)\cap \{\psi>c\}}|\partial_t \psi^m|\,(u-c)\,\mathrm {d}x\mathrm{d}t
\\&\leq 2^{\frac{1-m}{m}}\iint_{Q_{r_2,s_2}(z_0)\cap \{\psi>c\}}|\partial_t \psi^m|\,(u^m-c^m)^{\frac{1}{m}}\,\mathrm {d}x\mathrm{d}t
+
2^{\frac{1-m}{m}}\iint_{Q_{r_2,s_2}(z_0)\cap \{\psi>c\}}c|\partial_t \psi^m|\,\mathrm {d}x\mathrm{d}t
\\&\leq \frac{m2^{\frac{1-m}{m}}}{m+1}\iint_{Q_{r_2,s_2}(z_0)}|\partial_t \psi^m|^{\frac{m+1}{m}}\chi_{\{u>c\}}\,\mathrm {d}x\mathrm{d}t
+\frac{2^{\frac{1-m}{m}}}{m}\iint_{Q_{r_2,s_2}(z_0)}(u^m-c^m)_+^{\frac{m+1}{m}}\,\mathrm {d}x\mathrm{d}t
\\&\quad+
2^{\frac{1-m}{m}}\iint_{Q_{r_2,s_2}(z_0)\cap \{\psi>c\}}\psi\,|\partial_t \psi^m|\,\mathrm {d}x\mathrm{d}t,
\end{split}\end{equation*}
where we have used Young's inequality for the last estimate. Since $0<m<1$, we have
$(u^m-c^m)_+^{\frac{m+1}{m}}\leq (u-c)_+(u^m-c^m)_+$. In view of $0<s_2-s_1<64R^2\leq 64$, we use Young's inequality to obtain
\begin{equation}\begin{split}\label{VII}
\uppercase\expandafter{\romannumeral7}&\leq \gamma \iint_{Q_{r_2,s_2}(z_0)}|\partial_t \psi^m|^{\frac{m+1}{m}}\chi_{\{u>c\}}\,\mathrm {d}x\mathrm{d}t
+\gamma \iint_{Q_{r_2,s_2}(z_0)}\psi^{m+1}\chi_{\{u>c\}}\,\mathrm {d}x\mathrm{d}t
\\&\quad+\frac{\gamma}{s_2-s_1}\iint_{Q_{r_2,s_2}(z_0)}(u-c)_+(u^m-c^m)_+\,\mathrm {d}x\mathrm{d}t,
\end{split}\end{equation}
where the constant $\gamma$ depends only upon $m$.
Our next aim is to find a lower bound for $\uppercase\expandafter{\romannumeral8}$. We fix $t_1\in \Lambda_{s_1}(t_0)$ and
consider the superlevel set $\{B_{r_2}(x_0): u(x,t_1)\geq \psi_c (x,t_1) \}$. On this set, $u\geq c$ and there holds
\begin{equation}\begin{split}\label{L1L2L3}
(u&-c)(u^m-c^m)\\&=(u-\psi_c)(u^m-\psi_c^m)+(u-\psi_c)(\psi_c^m-c^m)
\\&\quad+(\psi_c-c)(u^m-\psi_c^m)+(\psi_c-c)(\psi_c^m-c^m)
\\&\leq(u-\psi_c)(u^m-\psi_c^m)+(\psi-c)_+(\psi^m-c^m)_+
\\&\quad+(u-c)(\psi^m-c^m)_++(\psi-c)_+(u^m-c^m).
\end{split}\end{equation}
Denote $L_1=(u-c)(\psi^m-c^m)_+$ and $L_2=(\psi-c)_+(u^m-c^m)$.
To estimate $L_1$, we first consider the easy case $(\psi^m-c^m)_+\leq4^{-1}(u^m-c^m)$.
In this case, we get
\begin{equation}\begin{split}\label{L1-1}L_1\leq \frac{1}{4}(u-c)(u^m-c^m).\end{split}\end{equation}
While in the case $(\psi^m-c^m)_+>4^{-1}(u^m-c^m)$, we have $\psi\geq c$ and $u^m<4\psi^m-3c^m\leq 4\psi^m$. Since $\frac{1}{m}>1$, we find that
\begin{equation*}\begin{split}
u-c&=|(u^m)^{\frac{1}{m}}-(c^m)^{\frac{1}{m}}|
\\&\leq \gamma ((u^m)^{\frac{1}{m}-1}+(c^m)^{\frac{1}{m}-1})(u^m-c^m)
\\&\leq \gamma (\psi^{1-m}+c^{1-m})(\psi^m-c^m)
\\&=\gamma ((\psi^m)^{\frac{1}{m}-1}+(c^m)^{\frac{1}{m}-1})(\psi^m-c^m)
\\&\leq \gamma ((\psi^m)^{\frac{1}{m}}-(c^m)^{\frac{1}{m}})=\gamma (\psi-c)=\gamma (\psi-c)_+,
\end{split}\end{equation*}
where the constant $\gamma$ depends only on $m$. Combining this estimate with \eqref{L1-1}, we obtain
\begin{equation}\begin{split}\label{L1}L_1\leq \frac{1}{4}(u-c)(u^m-c^m)+\gamma (\psi-c)_+(\psi^m-c^m)_+.\end{split}\end{equation}
Next, we consider the estimate of $L_2$. In the case $(\psi-c)_+\leq 4^{-1}(u-c)$, we have
\begin{equation}\begin{split}\label{L2-1}L_2\leq \frac{1}{4}(u-c)(u^m-c^m).\end{split}\end{equation}
In the case $(\psi-c)_+> 4^{-1}(u-c)$, we see that $\psi\geq c$ and $u<4\psi-3c$. Furthermore, we conclude that there exists $\gamma=\gamma(m)$ such that
\begin{equation*}\begin{split}
u^m-c^m&\leq (4\psi-3c)^m-c^m
\\&\leq 4\gamma\left[(4\psi-3c)+c\right]^{m-1}(\psi-c)
\\&=4\gamma(4\psi-2c)^{m-1}(\psi-c)
\\&\leq \gamma(2\psi)^{m-1}(\psi-c)
\\&\leq \gamma(\psi+c)^{m-1}(\psi-c)\leq \gamma(\psi^m-c^m).
\end{split}\end{equation*}
Combining this estimate with \eqref{L2-1}, we have shown that the estimate
\begin{equation}\begin{split}\label{L2}L_2\leq \frac{1}{4}(u-c)(u^m-c^m)+\gamma (\psi-c)_+(\psi^m-c^m)_+\end{split}\end{equation}
holds in any case.
Therefore, we conclude from \eqref{L1L2L3}, \eqref{L1} and \eqref{L2} that the inequality
\begin{equation*}\begin{split}(u-c)(x,t_1)\,(u^m-c^m)(x,t_1)\leq &2(u-\psi_c)(x,t_1)\,(u^m-\psi_c^m)(x,t_1)\\&+4\gamma (\psi-c)_+(x,t_1)\,(\psi^m-c^m)_+(x,t_1)\end{split}\end{equation*}
holds for any $x\in\{B_{r_2}(x_0): u(x,t_1)\geq \psi_c (x,t_1) \}$.
We now turn our
attention to the estimate of $\uppercase\expandafter{\romannumeral8}$.
It follows from \eqref{set identity} that
\begin{equation}\begin{split}\label{VIII-1}&\int_{B_{r_2}(x_0)\cap \{u\geq c\}}\phi^2(u-c)(x,t_1)\,(u^m-c^m)(x,t_1)\,\mathrm {d}x
\\&=\int_{B_{r_2}(x_0)\cap \{u\geq \psi_c\}}\phi^2(u-c)(x,t_1)\,(u^m-c^m)(x,t_1)\,\mathrm {d}x
\\&\leq 2\int_{B_{r_2}(x_0)\cap \{u\geq \psi_c\}}\phi^2(u-\psi_c)(x,t_1)\,(u^m-\psi_c^m)(x,t_1)\,\mathrm {d}x
\\ &\quad+4\gamma \int_{B_{r_2}(x_0)} (\psi-c)_+(x,t_1)\,(\psi^m-c^m)_+(x,t_1)\,\mathrm {d}x,
\end{split}\end{equation}
since $\phi\leq1$.
It remains to treat the second term on the right-hand side of \eqref{VIII-1}. For $t_1\in \Lambda_{s_1}(t_0)$, we decompose
\begin{equation*}\begin{split}
\int_{B_{r_2}(x_0)}& (\psi-c)_+(x,t_1)\,(\psi^m-c^m)_+(x,t_1)\,\mathrm {d}x\\&=
\int_{B_{r_2}(x_0)} (\psi-c)_+(\cdot,t_1)\,(\psi^m-c^m)_+(\cdot,t_1)\,\mathrm {d}x-
\dashint_{\Lambda_{s_2}(t_0)}\int_{B_{r_2}(x_0)} (\psi-c)_+\,(\psi^m-c^m)_+\,\mathrm {d}x\mathrm {d}t
\\&\quad+\dashint_{\Lambda_{s_2}(t_0)}\int_{B_{r_2}(x_0)} (\psi-c)_+\,(\psi^m-c^m)_+\,\mathrm {d}x\mathrm {d}t
\\&=:\uppercase\expandafter{\romannumeral8}_1+\uppercase\expandafter{\romannumeral8}_2,
\end{split}\end{equation*}
with the obvious meaning of $\uppercase\expandafter{\romannumeral8}_1$ and $\uppercase\expandafter{\romannumeral8}_2$.
We first observe that
\begin{equation*}\begin{split}
\uppercase\expandafter{\romannumeral8}_2\leq \frac{1}{2s_2}\iint_{Q_{r_2,s_2}}(\psi-c)_+\,(\psi^m-c^m)_+\,\mathrm {d}x\mathrm {d}t
\leq \frac{1}{s_2-s_1}\iint_{Q_{r_2,s_2}}(u-c)_+\,(u^m-c^m)_+\,\mathrm {d}x\mathrm {d}t,
\end{split}\end{equation*}
since $(\psi-c)_+\leq (u-c)_+$ and $(\psi^m-c^m)_+\leq (u^m-c^m)_+$. To estimate $\uppercase\expandafter{\romannumeral8}_1$, we note
 that the obstacle function $\psi$ is differentiable in time variable. Since $\partial_t\psi=\partial_t(\psi^m)^{\frac{1}{m}}=\frac{1}{m}
 \psi^{1-m}\partial_t\psi^m$ and $\{\psi>c\}\subset \{u>c\}$, we have
\begin{equation*}\begin{split}
\uppercase\expandafter{\romannumeral8}_1&\leq \dashint_{\Lambda_{s_2}(t_0)}\int_{B_{r_1}(x_0)}\big| \int_{t_1}^t\partial_\tau\left[(\psi-c)_+\,(\psi^m-c^m)_+
\right]\,\mathrm {d}\tau\,\big|\,\mathrm {d}x\mathrm {d}t
\\&\leq \iint_{Q_{r_2,s_2}} |\partial_t(\psi-c)_+|\,(\psi^m-c^m)_+  \,\mathrm {d}x\mathrm {d}t
\\ &\quad+ \iint_{Q_{r_2,s_2}} |\partial_t(\psi^m-c^m)_+|\,(\psi-c)_+  \,\mathrm {d}x\mathrm {d}t
\\&\leq \gamma\iint_{Q_{r_2,s_2}} |\partial_t\psi^m|^{\frac{m+1}{m}}\chi_{\{\psi>c\}}\,\mathrm {d}x\mathrm {d}t
+\gamma\iint_{Q_{r_2,s_2}} \psi^{m+1}\chi_{\{\psi>c\}}\,\mathrm {d}x\mathrm {d}t
\\&\quad +\frac{1}{m}\iint_{Q_{r_2,s_2}} |\partial_t\psi^m|\,\psi^{1-m}(\psi^m-c^m)_+  \,\mathrm {d}x\mathrm {d}t
\\&\leq \gamma\iint_{Q_{r_2,s_2}} |\partial_t\psi^m|^{\frac{m+1}{m}}\chi_{\{u>c\}}\,\mathrm {d}x\mathrm {d}t
+\gamma\iint_{Q_{r_2,s_2}} \psi^{m+1}\chi_{\{u>c\}}\,\mathrm {d}x\mathrm {d}t,
\end{split}\end{equation*}
where the constant $\gamma$ depends only upon $m$. Combining the estimates obtained for $\uppercase\expandafter{\romannumeral8}_1$ and $\uppercase\expandafter{\romannumeral8}_2$,
we deduce from \eqref{VIII-1} the estimate
\begin{equation}\begin{split}\label{VIII}\frac{1}{4}\int_{B_{r_2}(x_0)\cap \{u\geq c\}}&\phi^2(u-c)(x,t_1)\,(u^m-c^m)(x,t_1)\,\mathrm {d}x
\\&\leq \uppercase\expandafter{\romannumeral8}+\frac{1}{s_2-s_1}\iint_{Q_{r_2,s_2}}(u-c)_+\,(u^m-c^m)_+\,\mathrm {d}x\mathrm {d}t
\\ &\quad+\gamma\iint_{Q_{r_2,s_2}} |\partial_t\psi^m|^{\frac{m+1}{m}}\chi_{\{u>c\}}\,\mathrm {d}x\mathrm {d}t
+\gamma\iint_{Q_{r_2,s_2}} \psi^{m+1}\chi_{\{u>c\}}\,\mathrm {d}x\mathrm {d}t.
\end{split}\end{equation}
From \eqref{time4}-\eqref{VII} and \eqref{VIII}, we are led to the conclusion that there exists a constant
$\gamma=\gamma(m)$ such that
\begin{equation}\begin{split}\label{time final}
\lim\sup_{\epsilon\downarrow 0}\lim\sup_{h\downarrow 0}\,\langle \!\langle&\partial_t u,\alpha\eta (v^m-u^m)\rangle \!\rangle\\&\leq
\gamma \iint_{Q_{r_2,s_2}(z_0)}|\partial_t \psi^m|^{\frac{m+1}{m}}\chi_{\{u>c\}}\,\mathrm {d}x\mathrm{d}t
+\gamma \iint_{Q_{r_2,s_2}(z_0)}\psi^{m+1}\chi_{\{u>c\}}\,\mathrm {d}x\mathrm{d}t
\\&\quad+\frac{\gamma}{s_2-s_1}\iint_{Q_{r_2,s_2}(z_0)}(u-c)_+(u^m-c^m)_+\,\mathrm {d}x\mathrm{d}t
\\&\quad-\frac{1}{4}\int_{B_{r_2}(x_0)\cap \{u\geq c\}}\phi^2(u-c)(x,t_1)\,(u^m-c^m)(x,t_1)\,\mathrm {d}x.
\end{split}\end{equation}
Another step in the proof of \eqref{Cac1} is to find an estimate for diffusion term in \eqref{weak identity}. We first note that
\begin{equation}\begin{split}\label{diffusion}
\lim\sup_{\epsilon\downarrow 0}\lim\sup_{h\downarrow 0}&\iint_{\Omega_T}\alpha A(x,t,u,Du^m)\cdot D(\eta(v^m-u^m))\,\mathrm {d}x\mathrm{d}t
\\&=-\iint_{\Omega_T\cap\{u>\psi_c\}}2\phi\zeta A(x,t,u,Du^m)\cdot (u^m-\psi_c^m)D\phi\,\mathrm {d}x\mathrm{d}t
\\&\quad -\iint_{\Omega_T\cap\{u>\psi_c\}}\phi^2\zeta A(x,t,u,Du^m)\cdot D(u^m-\psi_c^m)\,\mathrm {d}x\mathrm{d}t,
\end{split}\end{equation}
where
\begin{equation*}
	\zeta(t)=\begin{cases}
0,&\quad \text{for}\quad t\in (0,t_0-s_2],\\
	1-\frac{1}{s_2-s_1}(t_0-s_1-t),&\quad \text{for}\quad t\in (t_0-s_2,t_0-s_1],
\\
	1,&\quad \text{for}\quad t\in (t_0-s_1,t_1].
	\end{cases}
\end{equation*}
By Young's inequality and the growth assumption of the vector field $A$, we obtain the estimate for the first term on the right-hand
side
\begin{equation}\begin{split}\label{diffusion1}
\big|&\iint_{\Omega_T\cap\{u>\psi_c\}}2\phi\zeta A(x,t,u,Du^m)\cdot (u^m-\psi_c^m)D\phi\,\mathrm {d}x\mathrm{d}t\big|
\\&\leq 2\nu_1\iint_{\Omega_T\cap\{u>\psi_c\}}\phi\zeta|Du^m|(u^m-c^m)_+|D\phi|\,\mathrm {d}x\mathrm{d}t
\\&\leq \frac{\nu_0}{4}\iint_{\Omega_T\cap\{u>\psi_c\}}\phi^2\zeta|Du^m|^2\,\mathrm {d}x\mathrm{d}t
+\gamma \iint_{Q_{r_2,s_2}(z_0)}(u^m-c^m)_+^2|D\phi|^2\,\mathrm{d}x\mathrm{d}t,
\end{split}\end{equation}
where the constant $\gamma$ depends only upon $\nu_0$ and $\nu_1$.
Next, we consider the second term on the right-hand
side of \eqref{diffusion}. Using Young's inequality and the ellipticity
assumption of the vector field $A$, we deduce
\begin{equation}\begin{split}\label{diffusion2}
-&\iint_{\Omega_T\cap\{u>\psi_c\}}\phi^2\zeta A(x,t,u,Du^m)\cdot D(u^m-\psi_c^m)\,\mathrm {d}x\mathrm{d}t
\\ \leq &-\nu_0\iint_{\Omega_T\cap\{u>\psi_c\}}\phi^2\zeta|Du^m|^2\,\mathrm {d}x\mathrm{d}t
+ \nu_1\iint_{\Omega_T\cap\{u>\psi_c\}}\phi^2\zeta|Du^m||D\psi_c^m|\,\mathrm {d}x\mathrm{d}t
\\ \leq &-\frac{\nu_0}{2}\iint_{\Omega_T\cap\{u>\psi_c\}}\phi^2\zeta|Du^m|^2\,\mathrm {d}x\mathrm{d}t
+ \frac{\nu_1^2}{2\nu_0}\iint_{Q_{r_2,s_2}(z_0)}|D\psi^m|^2\chi_{\{u>c\}}\,\mathrm {d}x\mathrm{d}t.
\end{split}\end{equation}
Furthermore, we need to consider the estimate of the gradient on the superlevel set $\{u>c\}$.
Since $u\geq \psi$, we have
$$\{z\in \Omega_T:c<u(z)\leq\psi_c\}=\{z\in \Omega_T:u(z)>c\}\cap \{z\in \Omega_T:u(z)=\psi(z)\}$$
and therefore $Du^m=D\psi^m$ a.e. on $\{z\in \Omega_T:c<u(z)\leq\psi_c\}$. This implies that
\begin{equation}\begin{split}\label{diffusion3}
\iint_{\Omega_T}\phi^2\zeta&|D(u^m-c^m)_+|^2\,\mathrm {d}x\mathrm{d}t=
\iint_{\Omega_T\cap\{u>c\}}\phi^2\zeta|Du^m|^2\,\mathrm {d}x\mathrm{d}t\\&
=\iint_{\Omega_T\cap\{u>\psi_c\}}\phi^2\zeta|Du^m|^2\,\mathrm {d}x\mathrm{d}t
+\iint_{\Omega_T\cap\{c<u\leq \psi_c\}}\phi^2\zeta|Du^m|^2\,\mathrm {d}x\mathrm{d}t
\\&
=\iint_{\Omega_T\cap\{u>\psi_c\}}\phi^2\zeta|Du^m|^2\,\mathrm {d}x\mathrm{d}t
+\iint_{\Omega_T\cap\{u>c\}\cap\{u=\psi\}}\phi^2\zeta|D\psi^m|^2\,\mathrm {d}x\mathrm{d}t
\\&
\leq \iint_{\Omega_T\cap\{u>\psi_c\}}\phi^2\zeta|Du^m|^2\,\mathrm {d}x\mathrm{d}t
+\iint_{Q_{r_2,s_2}(z_0)}|D\psi^m|^2\chi_{\{u>c\}}\,\mathrm {d}x\mathrm{d}t.
\end{split}\end{equation}
Combining the estimates \eqref{diffusion}-\eqref{diffusion3}, we conclude that
\begin{equation*}\begin{split}
\lim\sup_{\epsilon\downarrow 0}\lim&\sup_{h\downarrow 0}\iint_{\Omega_T}\alpha A(x,t,u,Du^m)\cdot D(\eta(v^m-u^m))\,\mathrm {d}x\mathrm{d}t
\\&\leq -\frac{\nu_0}{4}\iint_{\Omega_T}\phi^2\zeta|D(u^m-c^m)_+|^2\,\mathrm {d}x\mathrm{d}t
+\gamma\iint_{Q_{r_2,s_2}(z_0)}(u^m-c^m)_+^2|D\phi|^2\,\mathrm{d}x\mathrm{d}t
\\&\quad+\gamma\iint_{Q_{r_2,s_2}(z_0)}|D\psi^m|^2\chi_{\{u>c\}}\,\mathrm {d}x\mathrm{d}t.
\end{split}\end{equation*}
This estimate together with \eqref{time final} yield
\begin{equation*}\begin{split}\int_{B_{r_2}(x_0)}&\phi^2(u-c)_+(x,t_1)\,(u^m-c^m)_+(x,t_1)\,
\mathrm{d}x+\int_{t_0-s_1}^{t_1}\int_{B_{r_2}(x_0)}\phi^2|D(u^m-c^m)_+|^2\,\mathrm{d}x\mathrm{d}t
\\&\leq \frac{\gamma}{s_2-s_1}\iint_{Q_{r_2,s_2}(z_0)}(u-c)_+(u^m-c^m)_+\,
\mathrm{d}x\mathrm{d}t
\\&\quad+\gamma\iint_{Q_{r_2,s_2}(z_0)}(u^m-c^m)_+^2|D\phi|^2\,\mathrm{d}x\mathrm{d}t
\\&\qquad+\gamma\iint_{Q_{r_2,s_2}(z_0)}\left(\psi^{m+1}+|\partial_t\psi^m|^{\frac{m+1}{m}}+|D\psi^m|^2\right)\chi_{\{u>c\}}\,
\mathrm{d}x\mathrm{d}t\end{split}\end{equation*}
for any $t_1\in \Lambda_{s_1}(t_0)$. This proves the desired estimate \eqref{Cac1} by taking the supremum over
$t_1\in \Lambda_{s_1}(t_0)$ in the first term and $t_1=t_0+s_1$ in the second one.

Finally, we come to the proof of \eqref{Cac2}.
This result will be proved if we can show that the estimate
\begin{equation}\begin{split}\label{Cac2 t1}\int_{B_{r_2}(x_0)}\phi^2(u-c)_-&(x,t_1)\,(u^m-c^m)_-(x,t_1)\,
\mathrm{d}x+\int_{t_0-s_1}^{t_1}\int_{B_{r_2}(x_0)}\phi^2|D(u^m-c^m)_-|^2\,\mathrm{d}x\mathrm{d}t
\\&\leq \frac{\gamma}{s_2-s_1}\iint_{Q_{r_2,s_2}(z_0)}(u-c)_-(u^m-c^m)_-\,
\mathrm{d}x\mathrm{d}t\\&\qquad+\gamma \iint_{Q_{r_2,s_2}(z_0)}(u^m-c^m)_-^2|\,D\phi\,|^2\,\mathrm{d}x\mathrm{d}t,
\end{split}\end{equation}
holds for any $t_1\in \Lambda_{s_1}(t_0)$. In order to prove this estimate, we will work on the sublevel set $\{u<c\}$ and the argument is similar in spirit to
\cite[Lemma 3.1 (ii)]{CS18} and \cite[Lemma 4.1]{GS}.

According to the proof of \cite[Lemma 3.1 (ii)]{CS18}, we set
\begin{equation*}v^m=[\![u^m]\!]_h+([\![u^m]\!]_h-c^m)_-+\|\psi^m-[\![\psi^m]\!]_h\|_{L^\infty(\Omega_T)}\end{equation*}
as a comparison map and obtain
\begin{equation*}\begin{split}
-&\iint_{\Omega_T}\eta\alpha u\partial_tv^m\,\mathrm {d}x\mathrm{d}t
\\&\leq\iint_{\Omega_T}\eta\alpha^\prime\frac{m}{m+1} [\![u^m]\!]_h^{\frac{m+1}{m}}
\,\mathrm {d}x\mathrm{d}t+\iint_{\Omega_T}\eta\alpha^\prime [\![u^m]\!]_h^{\frac{1}{m}}([\![u^m]\!]_h-c^m)_-\,\mathrm {d}x\mathrm{d}t
\\&\quad+\iint_{\Omega_T}\eta\alpha^\prime\int_{[\![u^m]\!]_h^{\frac{1}{m}}}^{c}(c^m-y^m)_+\,\mathrm {d}y\,\mathrm {d}x\mathrm{d}t,
\end{split}\end{equation*}
where the cut-off function $\alpha$ is defined in \eqref{alpha} and $\eta=\phi^2$.
To estimate the third term on the right-hand side, we infer from \eqref{auxiliary2} that
\begin{equation*}\begin{split}\frac{m}{2}([\![u^m]\!]_h^{\frac{1}{m}}-c)_-([\![u^m]\!]_h-c^m)_-\leq
\int_{[\![u^m]\!]_h^{\frac{1}{m}}}^{c}(c^m-y^m)_+\,\mathrm {d}y\leq ([\![u^m]\!]_h^{\frac{1}{m}}-c)_-([\![u^m]\!]_h-c^m)_-.
\end{split}\end{equation*}
At this point, the desired estimate \eqref{Cac2 t1} follows from a standard argument (see for instance \cite[page 26-28]{GS} and \cite[page 12]{CS18}) and
we omit the details.
The proof of the lemma is now complete.
\end{proof}
\section{Estimates on the spatial average}
This section is devoted to the study of a gluing Lemma, which concerns weighted mean values of the weak solution on
different time slices. We first state and prove the
gluing lemma on the condition that the functions $\Psi$ and $\partial_t \psi^{1-m}$ are locally integrable.
Let $B$ be an open ball in $\Omega\subset \mathbb{R}^n$ and let $\eta\geq 0$ be a smooth function supported in the compact set $\bar B$.
Here and subsequently,
we define 
\begin{equation*}\begin{split}(u(t))^\eta_B=\frac{1}{\int_B\eta\,\mathrm {d}x}\int_{B}u(x,t)\eta(x)\,\mathrm {d}x.\end{split}\end{equation*}
The following lemma is our main result in this section.
\begin{lemma}\label{gluing lemma} Let $u$ be a nonnegative weak solution to the obstacle problem in the sense of Definition \ref{definition}. Fix a point $z_0=(x_0,t_0)\in\Omega_T$ and assume that $Q_{r_1,s}(z_0)\subset Q_{r_2,s}(z_0)
\subset \Omega_T$. Let $\xi\in C_0^\infty (B_{r_2}(x_0))$, $0\leq \xi\leq1$ in $B_{r_2}(x_0)$,
$ \xi\equiv 1$ in $B_{r_1}(x_0)$ and $|D\xi|\leq 2(r_2-r_2)^{-1}$.
Let $\Psi_1$ be the quantity
\begin{equation*}\Psi_1=\left[\left(\ \fiint_{Q_{r_2,s}(z_0)}\Psi
\,\mathrm {d}x\mathrm{d}t\right)^{\frac{1}{m+1}} +\left(\ \fiint_{Q_{r_2,s}(z_0)}|\partial_t \psi^{1-m}|\,\mathrm {d}x\mathrm {d}t\right)^{\frac{1}{1-m}}\right].\end{equation*}
Then for any $t_1$, $t_2\in \Lambda_{s}(t_0)$, there holds
\begin{equation}\begin{split}\label{gluing 1}|(u(t_1))^\xi_{B_{r_2}(x_0)}&-(u(t_2))^\xi_{B_{r_2}(x_0)}|
\leq \gamma\left(\frac{s}{r_2-r_1}\right)\ \fiint_{Q_{r_2,s}(z_0)}|Du^m|\,\mathrm {d}x\mathrm{d}t+\gamma \left(\frac{s}{r_2-r_1}\right)\Psi_1
\\&+\gamma \left(\frac{s}{r_2-r_1}\right)^m\left(\ \fiint_{Q_{r_2,s}(z_0)}\psi^{1-m}\,\mathrm {d}x\mathrm{d}t\right)
\left(\ \fiint_{Q_{r_2,s}(z_0)}\Psi
\,\mathrm {d}x\mathrm{d}t\right)^{\frac{m}{m+1}}
\\&\quad+\gamma
 s\left(\frac{r_2-r_1}{s}\right)^m\left(\ \fiint_{Q_{r_2,s}(z_0)}\Psi \,\mathrm {d}x\mathrm{d}t\right)^{\frac{1}{m+1}}
\end{split}\end{equation}
and
\begin{equation}\begin{split}\label{gluing 2}|(u(t_1))^\xi_{B_{r_2}(x_0)}&-(u(t_2))^\xi_{B_{r_2}(x_0)}|
\leq \gamma\left(\frac{s}{r_2-r_1}\right)\ \fiint_{Q_{r_2,s}(z_0)}|Du^m|\,\mathrm {d}x\mathrm{d}t+\gamma \left(\frac{s}{r_2-r_1}\right)\Psi_1
\\&+\gamma r_2\left(\ \fiint_{Q_{r_2,s}(z_0)}\psi^{1-m}\,\mathrm {d}x\mathrm{d}t\right)\left(\ \fiint_{Q_{r_2,s}(z_0)}\Psi
 \,\mathrm {d}x\mathrm{d}t\right)^{\frac{m}{m+1}}
 \\&\quad+\gamma r_2^{\frac{1}{m}}\left(\ \fiint_{Q_{r_2,s}(z_0)}\Psi \,\mathrm {d}x\mathrm{d}t\right)^{\frac{1}{m+1}},
\end{split}\end{equation}
where the constat $\gamma$ depends only upon $\nu_0$, $\nu_1$ and $m$.
\end{lemma}
\begin{proof} Our proof is in the spirit of \cite[Lemma 3.2, Lemma 4.1]{CS19}. Without loss of generality, we may assume that $t_1<t_2$.
In the variational inequality \eqref{weak identity} we choose $\eta=\xi$ as a cut-off function in space and, motivated by the proof of
of \cite[Lemma 3.2]{CS19}, we choose
\begin{equation*}
	\alpha(t)=\begin{cases}
0,&\quad \text{for}\quad t\in (0,t_1-\epsilon),\\
	1+\frac{1}{\epsilon}(t-t_1),&\quad \text{for}\quad t\in [t_1-\epsilon,t_1),
\\
	1,&\quad \text{for}\quad t\in [t_1,t_2],\\
	1-\frac{1}{\epsilon}(t-t_2),&\quad \text{for}\quad t\in (t_2,t_2+\epsilon],
\\ 0,&\quad \text{for}\quad t\in (t_2+\epsilon,T),
	\end{cases}
\end{equation*}
as a cut-off function in time, where $0<\epsilon\ll 1$. Next, we distinguish between the cases $(u(t_1))^\xi_{B_{r_2}(x_0)}\geq (u(t_2))^\xi_{B_{r_2}(x_0)}$ and $(u(t_1))^\xi_{B_{r_2}(x_0)}<(u(t_2))^\xi_{B_{r_2}(x_0)}$. In the first case, the argument in
\cite[page 19]{CS19} actually shows that
\begin{equation*}\begin{split}|(u(t_1))^\xi_{B_{r_2}(x_0)}-(u(t_2))^\xi_{B_{r_2}(x_0)}|
=(u(t_1))^\xi_{B_{r_2}(x_0)}-(u(t_2))^\xi_{B_{r_2}(x_0)}
&\leq \frac{\gamma s}{r_2-r_1}\ \fiint_{Q_{r_2,s}(z_0)}|Du^m|\,\mathrm {d}x\mathrm{d}t
\end{split}\end{equation*}
by choosing $v^m=\max\{[\![u^m]\!]_h+1,\psi^m\}$ as a comparison map in \eqref{weak identity}.

It suffices to prove the lemma in the case
$(u(t_1))^\xi_{B_{r_2}(x_0)}<(u(t_2))^\xi_{B_{r_2}(x_0)}$. Let $\mu$ be a fixed positive constant, which will be determined later.
We follow the argument in \cite[page 14-15]{CS19} to choose
\begin{equation*}v^m=\max\{[\![u^m]\!]_h-\mu^m,\psi^m\}\end{equation*}
as a comparison map in \eqref{weak identity} and deduce
\begin{equation*}\begin{split}\label{gluing time}
\,\langle \!\langle\partial_t u,\alpha\eta (v^m-u^m)\rangle \!\rangle&\leq I_h+
\mu^m\iint_{\Omega_T}\xi \alpha^\prime u\,\mathrm {d}x\mathrm{d}t+L,\end{split}\end{equation*}
where
we abbreviated
\begin{equation*}\begin{split}
L=&-\iint_{\Omega_T}\xi\alpha^\prime u(\psi^m+\mu^m-[\![u^m]\!]_h)_+\,\mathrm {d}x\mathrm{d}t
\\&+\iint_{\Omega_T\cap\{[\![u^m]\!]_h\leq \psi^m+\mu^m\}}\xi\alpha [\![u^m]\!]_h^{\frac{1}{m}}(\partial_t[\![u^m]\!]_h-\partial_t
\psi^m)\,\mathrm {d}x\mathrm{d}t
\end{split}\end{equation*}
and the term $I_h$ tends to zero as $h\downarrow 0$. To estimate $L$, we use integration by parts to obtain
\begin{equation*}\begin{split}
L&=-\iint_{\Omega_T}\xi\alpha^\prime u(\psi^m+\mu^m-[\![u^m]\!]_h)_+\,\mathrm {d}x\mathrm{d}t
\\&\quad-\iint_{\Omega_T}\xi\alpha [\![u^m]\!]_h^{\frac{1}{m}}\partial_t(\psi^m+\mu^m-[\![u^m]\!]_h)_+\,\mathrm {d}x\mathrm{d}t
\\&=\iint_{\Omega_T}\xi\alpha^\prime ([\![u^m]\!]_h^{\frac{1}{m}}-u)(\psi^m+\mu^m-[\![u^m]\!]_h)_+\,\mathrm {d}x\mathrm{d}t
\\&\quad+\iint_{\Omega_T}\xi\alpha \partial_t\big([\![u^m]\!]_h^{\frac{1}{m}}\big)(\psi^m+\mu^m-[\![u^m]\!]_h)_+\,\mathrm {d}x\mathrm{d}t
\\&=:L_1+L_2,
\end{split}\end{equation*}
with the obvious meaning of $L_1$ and $L_2$. By Lebesgue's dominated
convergence theorem, we see that $L_1$ tends to zero as $h\downarrow 0$. Next, we consider the estimate for $L_2$.
Noting that
\begin{equation*}\begin{split}
\frac{\partial}{\partial t}&\big[\int_{[\![u^m]\!]_h^{\frac{1}{m}}}^{(\mu^m+\psi^m)^{\frac{1}{m}}}(\mu^m+\psi^m-y^m)_+\,\mathrm {d}y\big]
\\&=
-\partial_t\big([\![u^m]\!]_h^{\frac{1}{m}}\big)(\mu^m+\psi^m-[\![u^m]\!]_h)_++
\partial_t \psi^m((\mu^m+\psi^m)^{\frac{1}{m}}-[\![u^m]\!]_h^{\frac{1}{m}})_+,
\end{split}\end{equation*}
we use integration by parts to obtain
\begin{equation*}\begin{split}
L_2=&\iint_{\Omega_T}\xi\alpha^\prime\int_{[\![u^m]\!]_h^{\frac{1}{m}}}^{(\mu^m+\psi^m)^{\frac{1}{m}}}(\mu^m+\psi^m-y^m)_+\,\mathrm {d}y
\,\mathrm {d}x\mathrm{d}t
\\&\quad
+\iint_{\Omega_T}\xi\alpha\partial_t \psi^m((\mu^m+\psi^m)^{\frac{1}{m}}-[\![u^m]\!]_h^{\frac{1}{m}})_+\,\mathrm {d}x\mathrm{d}t
\\=&:L_3+L_4,
\end{split}\end{equation*}
with the obvious meaning of $L_3$ and $L_4$. From Lemma \ref{auxiliary}, there exists a constant $\gamma=\gamma(m)$ such that
\begin{equation*}\begin{split}\lim\sup_{\epsilon\downarrow 0}\lim\sup_{h\downarrow 0} L_3&\leq
\int_{B_{r_2}(x_0)} ((\mu^m+\psi^m)^{\frac{1}{m}}-u)_+(x,t_1)\,(\mu^m+\psi^m-u^m)_+(x,t_1)\,\mathrm {d}x
\\&\leq \gamma \mu^{2m}\int_{B_{r_2}(x_0)}\psi(x,t_1)^{1-m}\mathrm {d}x+\gamma \mu^{m+1},
\end{split}\end{equation*}
since
\begin{equation}\begin{split}\label{mu u}((\mu^m+\psi^m)^{\frac{1}{m}}-u)_+&\leq (\mu^m+\psi^m+u^m)^{\frac{1}{m}-1}
(\mu^m+\psi^m-u^m)_+
\\&\leq \gamma \mu^m(\mu^{1-m}+\psi^{1-m}).
\end{split}\end{equation}
Moreover, we note that
\begin{equation*}\begin{split}\int_{B_{r_2}(x_0)}\psi(x,t_1)^{1-m}\mathrm {d}x&=
\int_{B_{r_2}(x_0)}\psi(x,t_1)^{1-m}\mathrm {d}x-\dashint_{\Lambda_{s}(t_0)}\int_{B_{r_2}(x_0)}\psi(x,t)^{1-m}\,\mathrm {d}x\mathrm {d}t
\\& \quad +\dashint_{\Lambda_{s}(t_0)}\int_{B_{r_2}(x_0)}\psi(x,t)^{1-m}\,\mathrm {d}x\mathrm {d}t
\\&=\dashint_{\Lambda_{s}(t_0)}\int_{B_{r_2}(x_0)}\int_t^{t_1}\partial_\tau[\psi(x,\tau)^{1-m}]\,\mathrm {d}\tau
\mathrm {d}x\mathrm {d}t
\\& \quad +|B_{r_2}(x_0)|\ \fiint_{Q_{r_2,s}(z_0)}\psi^{1-m}\,\mathrm {d}x\mathrm {d}t
\\&\leq \iint_{Q_{r_2,s}(z_0)}|\partial_t \psi^{1-m}|\,\mathrm {d}x\mathrm {d}t+
|B_{r_2}(x_0)|\ \fiint_{Q_{r_2,s}(z_0)}\psi^{1-m}\,\mathrm {d}x\mathrm {d}t
\end{split}\end{equation*}
and this implies
\begin{equation*}\begin{split}\lim\sup_{\epsilon\downarrow 0}\lim\sup_{h\downarrow 0} L_3&\leq
\gamma \mu^{2m}\iint_{Q_{r_2,s}(z_0)}|\partial_t \psi^{1-m}|\,\mathrm {d}x\mathrm {d}t\\&+\gamma \mu^{2m}
|B_{r_2}(x_0)|\ \fiint_{Q_{r_2,s}(z_0)}\psi^{1-m}\,\mathrm {d}x\mathrm {d}t+\gamma \mu^{m+1}.
\end{split}\end{equation*}
Next, we consider the estimate for $L_4$.
From \eqref{mu u}, we deduce
\begin{equation*}\begin{split}\lim\sup_{\epsilon\downarrow 0}\lim\sup_{h\downarrow 0} L_4&\leq
\iint_{Q_{r_2,s}(z_0)}\mu |\partial_t \psi^m|\,\mathrm {d}x\mathrm{d}t+
\iint_{Q_{r_2,s}(z_0)}\mu^{m}\psi^{1-m}|\partial_t \psi^m|\,\mathrm {d}x\mathrm{d}t
\\&= \iint_{Q_{r_2,s}(z_0)}\mu |\partial_t \psi^m|\,\mathrm {d}x\mathrm{d}t+
2s\mu^{m}|B_{r_2}(x_0)|\ \fiint_{Q_{r_2,s}(z_0)}\psi^{1-m}|\partial_t \psi^m|\,\mathrm {d}x\mathrm{d}t.
\end{split}\end{equation*}
By H\"older's inequality, we obtain
\begin{equation*}\begin{split}\lim\sup_{\epsilon\downarrow 0}\lim\sup_{h\downarrow 0} L_4
&\leq \iint_{Q_{r_2,s}(z_0)}\mu |\partial_t \psi^m|\,\mathrm {d}x\mathrm{d}t\\&\quad+
2s\mu^{m}|B_{r_2}(x_0)|
\left(\ \fiint_{Q_{r_2,s}(z_0)}\psi\,\mathrm {d}x\mathrm{d}t\right)^{1-m}
\left(\ \fiint_{Q_{r_2,s}(z_0)}|\partial_t \psi^m|^{\frac{1}{m}}\,\mathrm {d}x\mathrm{d}t\right)^{m}
\\&\leq \iint_{Q_{r_2,s}(z_0)}\mu |\partial_t \psi^m|\,\mathrm {d}x\mathrm{d}t+
2s\mu^{m}|B_{r_2}(x_0)|
\left(\ \fiint_{Q_{r_2,s}(z_0)}\Psi\,\mathrm {d}x\mathrm{d}t\right)^{\frac{1}{m+1}}.
\end{split}\end{equation*}
To estimate the diffusion term, we infer from the argument in \cite[page 17]{CS19} that
\begin{equation*}\begin{split}
\lim&\sup_{\epsilon\downarrow 0}\lim\sup_{h\downarrow 0}
\iint_{\Omega_T}\alpha A(x,t,u,Du^m)\cdot D(\eta(v^m-u^m))\,\mathrm {d}x\mathrm{d}t\\&\leq
\gamma \mu^m\frac{1}{r_2-r_1}\ \iint_{Q_{r_2,s}(z_0)}|Du^m|\,\mathrm {d}x\mathrm{d}t+
\gamma \ \iint_{Q_{r_2,s}(z_0)}|D\psi^m|^2\,\mathrm {d}x\mathrm{d}t.
\end{split}\end{equation*}
Combining the estimates above, we conclude that
\begin{equation*}\begin{split}&(u(t_2))^\xi_{B_{r_2}(x_0)}-(u(t_1))^\xi_{B_{r_2}(x_0)}\\&\leq
\gamma \mu
+\gamma \mu^{m}
\ \fiint_{Q_{r_2,s}(z_0)}\psi^{1-m}\,\mathrm {d}x\mathrm {d}t
+\gamma s\mu^{m}\ \fiint_{Q_{r_2,s}(z_0)}|\partial_t \psi^{1-m}|\,\mathrm {d}x\mathrm {d}t
 \\&\quad+2s\ \fiint_{Q_{r_2,s}(z_0)}\mu^{1-m} |\partial_t \psi^m|\,\mathrm {d}x\mathrm{d}t+
2s\left(\ \fiint_{Q_{r_2,s}(z_0)}\Psi\,\mathrm {d}x\mathrm{d}t\right)^{\frac{1}{m+1}}
\\&\qquad+\gamma \frac{s}{r_2-r_1}\ \fiint_{Q_{r_2,s}(z_0)}|Du^m|\,\mathrm {d}x\mathrm{d}t
+\gamma \frac{s}{\mu^m}\ \fiint_{Q_{r_2,s}(z_0)}|D\psi^m|^2\,\mathrm {d}x\mathrm{d}t,
\end{split}\end{equation*}
where the constant $\gamma$ depends only upon $\nu_0$, $\nu_1$ and $m$. Applying Young's inequality,
we estimate the third and fourth terms on the right-hand side as follows:
\begin{equation*}\begin{split}
s\mu^{m}\ \fiint_{Q_{r_2,s}(z_0)}&|\partial_t \psi^{1-m}|\,\mathrm {d}x\mathrm {d}t+
2s\ \fiint_{Q_{r_2,s}(z_0)}\mu^{1-m} |\partial_t \psi^m|\,\mathrm {d}x\mathrm{d}t
\\&\leq 2s\mu+2s\left(\fiint_{Q_{r_2,s}(z_0)}|\partial_t \psi^{1-m}|\,\mathrm {d}x\mathrm {d}t\right)^{\frac{1}{1-m}}
+2s\ \fiint_{Q_{r_2,s}(z_0)}|\partial_t \psi^m|^{\frac{1}{m}}\,\mathrm {d}x\mathrm{d}t
\\&\leq 2\mu+2s\left(\fiint_{Q_{r_2,s}(z_0)}|\partial_t \psi^{1-m}|\,\mathrm {d}x\mathrm {d}t\right)^{\frac{1}{1-m}}
+2s\left(\ \fiint_{Q_{r_2,s}(z_0)}|\partial_t \psi^m|^{\frac{m+1}{m}}\,\mathrm {d}x\mathrm{d}t\right)^{\frac{1}{m+1}},
\end{split}\end{equation*}
where we have used H\"older's inequality for the last estimate. This implies that the following inequality
\begin{equation}\begin{split}\label{mumu}(u(t_2))^\xi_{B_{r_2}(x_0)}&-(u(t_1))^\xi_{B_{r_2}(x_0)}\leq
\gamma \mu
+\gamma \mu^{m}
\ \fiint_{Q_{r_2,s}(z_0)}\psi^{1-m}\,\mathrm {d}x\mathrm {d}t
 \\&+2s\left(\ \fiint_{Q_{r_2,s}(z_0)}|\partial_t \psi^{1-m}|\,\mathrm {d}x\mathrm {d}t\right)^{\frac{1}{1-m}}
+2s\left(\ \fiint_{Q_{r_2,s}(z_0)}\Psi\,\mathrm {d}x\mathrm{d}t\right)^{\frac{1}{m+1}}
\\&+\gamma \frac{s}{r_2-r_1}\ \fiint_{Q_{r_2,s}(z_0)}|Du^m|\,\mathrm {d}x\mathrm{d}t
+\gamma \frac{s}{\mu^m}\ \fiint_{Q_{r_2,s}(z_0)}|D\psi^m|^2\,\mathrm {d}x\mathrm{d}t
\end{split}\end{equation}
holds for any $\mu>0$.
At this stage, we set $0<\delta\ll 1$. In the estimate \eqref{mumu} we choose
\begin{equation*}\begin{split}\mu=\frac{s}{r_2-r_1}\left(\ \fiint_{Q_{r_2,s}(z_0)}\left(\delta+\psi^{m+1}+
|\partial_t \psi^m|^{\frac{m+1}{m}}+|D\psi^m|^2\right)
\,\mathrm {d}x\mathrm{d}t\right)^{\frac{1}{m+1}}.
\end{split}\end{equation*}
This concludes the estimate \eqref{gluing 1} by passing to the limit $\delta\downarrow 0$.
Finally, if we choose
\begin{equation*}\begin{split}\mu=r_2^{\frac{1}{m}}\left(\ \fiint_{Q_{r_2,s}(z_0)}\left(\delta+\psi^{m+1}+
|\partial_t \psi^m|^{\frac{m+1}{m}}+|D\psi^m|^2\right)
\,\mathrm {d}x\mathrm{d}t\right)^{\frac{1}{m+1}},
\end{split}\end{equation*}
then the desired estimate \eqref{gluing 2} follows by passing to the limit $\delta\downarrow 0$. This finishes the proof of the lemma.
\end{proof}
Moreover, if $\psi^m$ is locally Lipschitz continuous and $\partial_t\psi^{1-m}$ is locally bounded, then we can rewrite the estimates \eqref{gluing 1} and \eqref{gluing 2} in the following ready-to-use form.
\begin{corollary}\label{gluing corollary} Suppose that
\begin{equation*}\begin{split}
\sup_{Q_{r_2,s}(z_0)}(\ \Psi^{\frac{1}{m+1}}+|\partial_t\psi^{1-m}|^{\frac{1}{1-m}}\ )\leq M_0
\end{split}\end{equation*}
for some $M_0>1$.
Then for any $t_1$, $t_2\in\Lambda_s(t_0)$, there holds
\begin{equation}\begin{split}\label{gluing 3}|(u(t_1))^\xi_{B_{r_2}(x_0)}-&(u(t_2))^\xi_{B_{r_2}(x_0)}|
\leq \gamma\left(\frac{s}{r_2-r_1}\right)\left(\ \fiint_{Q_{r_2,s}(z_0)}|Du^m|\,\mathrm {d}x\mathrm{d}t+M_0\right)
\\&+\gamma \left(\frac{s}{r_2-r_1}\right)^m\left(\ \fiint_{Q_{r_2,s}(z_0)}\psi^{1-m}\,\mathrm {d}x\mathrm{d}t\right)
M_0^m
+\gamma
 s\left(\frac{r_2-r_1}{s}\right)^mM_0
\end{split}\end{equation}
and
\begin{equation}\begin{split}\label{gluing 4}|(u(t_1))^\xi_{B_{r_2}(x_0)}-(u(t_2))^\xi_{B_{r_2}(x_0)}|&
\leq \gamma\left(\frac{s}{r_2-r_1}\right)\left(\ \fiint_{Q_{r_2,s}(z_0)}|Du^m|\,\mathrm {d}x\mathrm{d}t+M_0\right)
\\&+\gamma r_2\left(\ \fiint_{Q_{r_2,s}(z_0)}\psi^{1-m}\,\mathrm {d}x\mathrm{d}t\right)M_0^m
+\gamma r_2^{\frac{1}{m}}M_0,
\end{split}\end{equation}
where the constant $\gamma$ depends only on $\nu_0$, $\nu_1$ and $m$.
\end{corollary}
This corollary is a direct consequence of Lemma \ref{gluing lemma} and the proof is omitted.
\section{Reverse H\"older-type inequalities}
The proof of the reverse H\"older inequalities on intrinsic cylinders follows from
the analysis of two complementary cases. Following \cite{GS}, we give the definitions of degenerate and non-degenerate regimes.
\begin{definition}\cite{GS}
Fix a point $z_0\in \Omega_T$ and suppose that $Q_{R,R^2}(z_0)\subset\Omega_T$. Let $\epsilon>0$ be a fixed number and
let $Q_s(z_0)$ be an intrinsic cylinder constructed in \S 3. We call a cylinder $Q_s(z_0)$ degenerate if and only if
\begin{equation}\begin{split}\label{degenerate}\left(\ \fiint_{Q_s(z_0)}\big|u^m-(u^m)_{Q_s(z_0)}\big|^{\frac{m+1}{m}}
\,\mathrm{d}x\mathrm{d}t\right)^{\frac{1}{m+1}}\geq \epsilon
\left(\ \fiint_{Q_s(z_0)}u^{m+1}\,\mathrm{d}x\mathrm{d}t\right)^{\frac{1}{m+1}}
\end{split}\end{equation}
holds true. Moreover, we call a cylinder $Q_s(z_0)$ non-degenerate if and only if the following inequality holds:
\begin{equation}\begin{split}\label{non degenerate}\left(\ \fiint_{Q_s(z_0)}\big|u^m-(u^m)_{Q_s(z_0)}\big|^{\frac{m+1}{m}}
\,\mathrm{d}x\mathrm{d}t\right)^{\frac{1}{m+1}}\leq \epsilon
\left(\ \fiint_{Q_s(z_0)}u^{m+1}\,\mathrm{d}x\mathrm{d}t\right)^{\frac{1}{m+1}}.
\end{split}\end{equation}
\end{definition}
Next, we consider separately the degenerate and non-degenerate case.
\subsection{The degenerate alternative}
This subsection deals with the degenerate case. We first establish a boundedness result analogue to \cite[Proposition 5.2]{GS}.
The local boundedness for weak solutions to the singular parabolic obstacle problems was first proved by Cho and Scheven \cite{CS18}.
Here, we present a mean value type estimate and our proof is in the spirit of \cite[Proposition 5.2]{GS}.
\begin{lemma}\label{sup}Let $u$ be a nonnegative weak solution to the obstacle problem in the sense of Definition \ref{definition}.
Fix a point $z_0\in \Omega_T$ and suppose that $Q_{R,R^2}(z_0)\subset\Omega_T$.
Let $0<s\leq \frac{1}{2}R^2$ and $r(2s)$ makes sense.
Assume that the cylinder $Q_{s}(z_0)$ is intrinsic and
\begin{equation}\begin{split}\label{psi theta}\sup_{Q_{2s}(z_0)}\Psi\leq
\frac{\theta_{s}(z_0)^{\frac{m+1}{1-m}}}{s}.\end{split}\end{equation}
Then there exists a constant $\gamma=\gamma(n,m,\nu_0,\nu_1)$ such that
\begin{equation}\begin{split}\label{boundedness}
\sup_{Q_{s}(z_0)}u\leq \gamma\left(\ \fiint_{Q_{2s}(z_0)}u^{m+1}\,\mathrm {d}x\mathrm{d}t\right)^{\frac{1}{m+1}}.
\end{split}\end{equation}
\end{lemma}
\begin{proof}
There is no loss of generality in assuming $z_0=(x_0,t_0)=(0,0)$.
For $j=0,1,2,\cdots$, set $s_j=s+2^{-j}s$, $r_j=r(s_j)$,
$B_j=B_{r_j}$ and $Q_j=Q_{r_j,s_j}$. We define a sequence of numbers $k_j^m=k^m-2^{-j}k^m$,
where $k>0$ is to be determined.
Let $\zeta_j=\zeta_j(x)$ be a smooth function such that $\zeta_j\in C_0^\infty(B_j)$, $0\leq \zeta_j\leq1$, $\zeta_j\equiv 1$ in $B_{j+1}$ and $|D\zeta_j|\leq 2(r_j-r_{j+1})^{-1}$.
We now apply the Caccioppoli estimate \eqref{Cac1} with $(c,\,\phi,\,Q_{r_1,s_1},\,Q_{r_2,s_2})$
replaced by $(k_{j+1},\,\zeta_j,\,Q_{j+1},\,Q_j)$ to obtain
\begin{equation*}\begin{split}\esssup_{-t_{j+1}<t<t_{j+1}}&\int_{B_j}[(u^m-k_{j+1}^m)_+\zeta_j]^{\frac{m+1}{m}}(x,t)\,
\mathrm{d}x+\int_{-t_{j+1}}^{t_{j+1}}\int_{B_j}|D[(u^m-k_{j+1}^m)_+\zeta_j]|^2\,\mathrm{d}x\mathrm{d}t
\\&\leq \frac{\gamma}{s_{j}-s_{j+1}}\iint_{Q_j}u^{m+1}\chi_{\{u>k_{j+1}\}}\,
\mathrm{d}x\mathrm{d}t\\&\quad+\frac{\gamma}{(r_j-r_{j+1})^2}\iint_{Q_j}(u^m-k_{j+1}^m)_+^2\,\mathrm{d}x\mathrm{d}t
+\gamma\iint_{Q_j}\Psi\chi_{\{u>k_{j+1}\}}\,
\mathrm{d}x\mathrm{d}t.\end{split}\end{equation*}
We first observe from Lemma \ref{subcylinder} (5) that all the cylinders $Q_j$ are intrinsic. Moreover, from Lemma \ref{subcylinder} (4) and the assumption \eqref{psi theta}, we deduce
\begin{equation*}\begin{split}\iint_{Q_j}\Psi\chi_{\{u>k_{j+1}\}}\,
\mathrm{d}x\mathrm{d}t&\leq \sup_{Q_j}\Psi\
|\{u>k_{j+1}\}\cap Q_j|
\\ &\leq \gamma \frac{2^{j\frac{m+1}{m}}}{s}
\left(\frac{\theta_{2s}}{k^{1-m}}\right)^{\frac{m+1}{1-m}}
\iint_{Q_j}(u^m-k_j^m)_+^{\frac{m+1}{m}}\,
\mathrm{d}x\mathrm{d}t.
\end{split}\end{equation*}
Then, we follow the argument in \cite[page 33-34]{GS} to impose a condition $k\geq \theta_{2s}^{\frac{1}{1-m}}$
and obtain
\begin{equation*}\begin{split}\esssup_{-t_{j+1}<t<t_{j+1}}&\int_{B_j}[(u^m-k_{j+1}^m)_+\zeta_j]^{\frac{m+1}{m}}(x,t)\,
\mathrm{d}x+\int_{-t_{j+1}}^{t_{j+1}}\int_{B_j}|D[(u^m-k_{j+1}^m)_+\zeta_j]|^2\,\mathrm{d}x\mathrm{d}t
\\&\leq\frac{\gamma2^{j\frac{3(m+1)}{m}}}{s}\iint_{Q_j}(u^m-k_{j}^m)_+^{\frac{m+1}{m}}\,\mathrm{d}x\mathrm{d}t.
\end{split}\end{equation*}
Consequently, we can apply the parabolic Sobolev inequality to $(u^m-k_{j+1}^m)_+\zeta_j$ on the cylinder $B_{j}\times(-t_{j+1},t_{j+1})$, which gives
\begin{equation}\begin{split}\label{Yj}Y_{j+1}\leq \gamma 2^{bj}\left(\frac{|Q_j|^{\frac{2m+2}{nqm}}}{s^{\frac{2m+nm+n+2}{nqm}}k^{(m+1)(1-\frac{m+1}{qm})}}\right)Y_j^{1+\frac{2m+2}{nqm}},
\end{split}\end{equation}
where
$$b=2^{\frac{4(m+1)}{m}(2+\frac{2(m+1)}{nqm})},\qquad
q=2\frac{n+\frac{m+1}{m}}{n}\qquad\text{and}\qquad Y_j=\fiint_{Q_{j}}(u^m-k_{j}^m)_+^{\frac{m+1}{m}}\,\mathrm{d}x\mathrm{d}t.$$
For more details on the proof of \eqref{Yj}, we refer the reader to \cite[page 34]{GS}.
According to the argument in \cite[page 34]{GS}, we obtain $Y_j\to 0$ as $j\to\infty$, provided that
$$k=\gamma\left(\ \fiint_{Q_0}u^{m+1}\,\mathrm {d}x\mathrm{d}t\right)^{\frac{1}{m+1}},$$
where $\gamma>1$ depends only upon $n$, $\nu_0$, $\nu_1$ and $m$. This proves \eqref{boundedness} and the proof of Lemma \ref{sup} is complete.
\end{proof}
We remark that the intrinsic condition for $Q_s(z_0)$ is necessary in the proof of Lemma \ref{sup}. This restricts us to work with
the intrinsic cylinders in the degenerate regime. With the help of Lemma \ref{sup}, we can now establish the reverse H\"older inequality for the degenerate regime.
\begin{proposition}\label{degenerate regime}
Let $u$ be a nonnegative weak solution to the obstacle problem in the sense of Definition \ref{definition}.
Fix a point $z_0\in \Omega_T$ and suppose that $Q_{R,R^2}(z_0)\subset\Omega_T$.
Let $0<s\leq \frac{1}{3}R^2$ and $r(3s)$ makes sense.
Assume that the cylinder $Q_{s}(z_0)$ is intrinsic and satisfies \eqref{degenerate}. Moreover, assume that $\psi^m$ is locally Lipschitz
continuous and
\begin{equation*}\begin{split}
\sup_{Q_{3s}(z_0)}(\ \Psi^{\frac{1}{m+1}}+|\partial_t\psi^{1-m}|^{\frac{1}{1-m}}\ )\leq M_0,
\end{split}\end{equation*}
for some $M_0>0$. Then there exists $q_1\in(\frac{1}{2},1)$, depending only upon $n$ and $m$, such that the following holds:
\begin{equation}\begin{split}\label{degenerate result} \ \fiint_{Q_s(z_0)}|Du^m|^2
\,\mathrm{d}x\mathrm{d}t \leq c_\epsilon
\left(\ \fiint_{Q_{3s}(z_0)}|Du^m|^{2q_1}\,\mathrm{d}x\mathrm{d}t\right)^{\frac{1}{q_1}}
+c_\epsilon M_0^2+1.
\end{split}\end{equation}
\end{proposition}
\begin{proof}
For abbreviation, we assume that $z_0=(x_0,t_0)=(0,0)$. Initially, we use \eqref{Cac2} from Lemma \ref{Caccioppoli}
to obtain
\begin{equation*}\begin{split}\frac{1}{s}\esssup_{t\in \Lambda_s}&\ \dashint_{B_{r(s)}}u^{m+1}\,
\mathrm{d}x+\ \fiint_{Q_s}|Du^m|^2\,\mathrm{d}x\mathrm{d}t
\\&\leq \frac{\gamma}{s}\ \fiint_{Q_{2s}}u^{m+1}\,
\mathrm{d}x\mathrm{d}t+\frac{\gamma}{(r(2s)-r(s))^2}\ \fiint_{Q_{2s}}u^{2m}\,\mathrm{d}x\mathrm{d}t
+\gamma\ \fiint_{Q_{2s}}\Psi\,
\mathrm{d}x\mathrm{d}t.\end{split}\end{equation*}
From Lemma \ref{subcylinder} (1), (2), (4) and H\"older's inequality, we obtain
\begin{equation*}\begin{split}\frac{1}{s}\esssup_{t\in \Lambda_s}&\ \dashint_{B_{r(s)}}u^{m+1}\,
\mathrm{d}x+\ \fiint_{Q_s}|Du^m|^2\,\mathrm{d}x\mathrm{d}t
\leq \gamma \frac{\theta_s^{\frac{m+1}{1-m}}}{s}
+\gamma\ \fiint_{Q_{2s}}\Psi\,
\mathrm{d}x\mathrm{d}t.\end{split}\end{equation*}
Before proceeding further, we distinguish between two cases:
\begin{equation*}\begin{split}
\frac{\theta_s^{\frac{m+1}{1-m}}}{s}\leq \sup_{Q_{2s}}\Psi\qquad\text{and}\qquad
\sup_{Q_{2s}}\Psi\leq \frac{\theta_s^{\frac{m+1}{1-m}}}{s}.
\end{split}\end{equation*}
Observe that the desired estimate \eqref{degenerate result} holds immediately in the first case. It remains to treat the second case.
We first note that
\begin{equation}\begin{split}\label{Cac3}\frac{1}{s}\esssup_{t\in \Lambda_s}&\ \dashint_{B_{r(s)}}u^{m+1}\,
\mathrm{d}x+\ \fiint_{Q_s}|Du^m|^2\,\mathrm{d}x\mathrm{d}t
\leq \gamma \frac{\theta_s^{\frac{m+1}{1-m}}}{s}.
\end{split}\end{equation}
Our next aim is to find an upper bound for $s^{-1}\theta_s^{\frac{m+1}{1-m}}$.
Let $\eta\in C_0^\infty (B_{r(3s)})$, $0\leq \eta\leq1$ in $B_{r(3s)}$,
$ \eta\equiv 1$ in $B_{r(2s)}$ and $|D\eta|\leq 2(r(3s)-r(2s))^{-1}$. We denote by $\lambda_0$ the constant
\begin{equation*}\begin{split}\lambda_0^m=\left(\ \dashint_{\Lambda_{2s}}(u(t))_{B_{r(3s)}}^\eta\,
\mathrm{d}t
\right)^m.\end{split}\end{equation*}
Since $\sup_{Q_{2s}}\Psi\leq s^{-1}\theta_s^{\frac{m+1}{1-m}}$, the assumptions of Lemma \ref{sup} are fulfilled.
Applying \eqref{boundedness} and H\"older's inequality, we obtain similar as in \cite[Corollary 5.4]{GS} that
\begin{equation}\begin{split}\label{lambda0}\theta_s^{\frac{m}{1-m}}\leq c\left(\ \fiint_{Q_{2s}}u
\,\mathrm{d}x\mathrm{d}t\right)^m\leq c\lambda_0^m.\end{split}\end{equation}
Next, we choose $q_1\in (\frac{1}{2},1)$ such that
\begin{equation}\begin{split}\label{q1}m>\frac{n-2q_1}{(2q_1-1)n+2q_1}.\end{split}\end{equation}
By Sobolev inequality and Lemma \ref{subcylinder} (4), we deduce
\begin{equation}\begin{split}\label{Sobolev}\dashint_{B_{r(3s)}}\big|u^m-(u^m(t))_{B_{r(3s)}}
\big|^{\frac{m+1}{m}}\,\mathrm{d}x\leq cr(s)^{\frac{m+1}{m}}\left(\ \dashint_{B_{r(3s)}}
|Du^m|^{2q_1}\,\mathrm{d}x\right)^{\frac{m+1}{2q_1m}}.\end{split}\end{equation}
Using the similar argument as in the proof of \cite[Proposition 6.2]{GS}, we infer from \eqref{degenerate}, \eqref{Cac3}, \eqref{lambda0}
and \eqref{Sobolev}
that
\begin{equation}\begin{split}\label{citedstep}
\left(\frac{\theta_s^{\frac{m+1}{1-m}}}{s}\right)^{\frac{\alpha}{q_1}}
&\leq c\frac{r(s)^2}{s^{\frac{\alpha}{q_1}}}\left(\ \fiint_{Q_{3s}(z_0)}|Du^m|^{2q_1}\,\mathrm{d}x\mathrm{d}t\right)^{\frac{1}{q_1}}
\\&\quad+\frac{c}{s^{\frac{\alpha}{q_1}}}\left(\ \dashint_{\Lambda_s}\theta_s^{-\frac{m+1}{m}}\big|(u(t))_{B_{r(3s)}}^\eta
-\lambda_0\big|^{\frac{m+1}{m}}\,\mathrm{d}t\right)^{\frac{\alpha}{q_1}},
\end{split}\end{equation}
where $\alpha=\frac{2q_1m}{m+1}$. To estimate the second term on the right-hand side, we apply the estimate \eqref{gluing 3}
from Corollary \ref{gluing corollary} to deduce
\begin{equation*}\begin{split}\big|(u(t))_{B_{r(3s)}}^\eta
-\lambda_0\big|&=\dashint_{\Lambda_{2s}}\big|(u(t))_{B_{r(3s)}}^\eta-(u(t))_{B_{r(3s)}}^\eta\big|\,\mathrm{d}\tau
\\&\leq \gamma\left(\frac{s}{r(3s)-r(2s)}\right)\left(\ \fiint_{Q_{3s}}|Du^m|\,\mathrm {d}x\mathrm{d}t+M_0\right)
\\&\quad+\gamma \left(\frac{s}{r(3s)-r(2s)}\right)^m\left(\ \fiint_{Q_{3s}}\psi^{1-m}\,\mathrm {d}x\mathrm{d}t\right)
M_0^m
+\gamma
 s\left(\frac{r(3s)-r(2s)}{s}\right)^mM_0
 \\&\leq \gamma\left(\frac{s}{r(s)}\right)\left(\ \fiint_{Q_{3s}}|Du^m|\,\mathrm {d}x\mathrm{d}t+M_0\right)
\\&\quad+\gamma \left(\frac{s}{r(s)}\right)^m\left(\ \fiint_{Q_{3s}}\psi^{1-m}\,\mathrm {d}x\mathrm{d}t\right)
M_0^m
+\gamma
 s\left(\frac{r(s)}{s}\right)^mM_0,
\end{split}\end{equation*}
where we have used Lemma \ref{subcylinder} (2), (4) for the last estimate. From this, we conclude that
\begin{equation*}\begin{split}
\frac{c}{s^{\frac{\alpha}{q_1}}}&\left(\ \dashint_{\Lambda_s}\theta_s^{-\frac{m+1}{m}}\big|(u(t))_{B_{r(3s)}}^\eta
-\lambda_0\big|^{\frac{m+1}{m}}\,\mathrm{d}t\right)^{\frac{\alpha}{q_1}}
\\&\leq \gamma \frac{1}{s^{\frac{\alpha}{q_1}}\theta_s^2}\left(\frac{s}{r(s)}\right)^2\left(\ \fiint_{Q_{3s}}|Du^m|\,\mathrm {d}x\mathrm{d}t+M_0\right)^2
\\&\quad+ \gamma \frac{1}{s^{\frac{\alpha}{q_1}}\theta_s^2}
\left(\frac{s}{r(s)}\right)^{2m}\left(\ \fiint_{Q_{3s}}\psi^{1-m}\,\mathrm {d}x\mathrm{d}t\right)^2
M_0^{2m}
+\gamma \frac{s^2}{s^{\frac{\alpha}{q_1}}\theta_s^2}\left(\frac{r(s)}{s}\right)^{2m}M_0^2
\\&= \gamma \frac{1}{s^{\frac{\alpha}{q_1}-1}\theta_s}\left(\ \fiint_{Q_{3s}}|Du^m|\,\mathrm {d}x\mathrm{d}t+M_0\right)^2
\\&\quad+ \gamma \frac{1}{s^{\frac{\alpha}{q_1}-m}\theta_s^{2-m}}
\left(\ \fiint_{Q_{3s}}\psi^{1-m}\,\mathrm {d}x\mathrm{d}t\right)^2
M_0^{2m}
+\gamma \frac{1}{s^{\frac{\alpha}{q_1}-2+m}\theta_s^{2+m}}M_0^2,
\end{split}\end{equation*}
since $s=\theta_sr(s)^2$. We insert this inequality in \eqref{citedstep} and this implies that
\begin{equation*}\begin{split}\frac{\theta_s^{\frac{m+1}{1-m}}}{s}&=\left(\frac{\theta_s^{\frac{m+1}{1-m}}}{s}\right)^{\frac{\alpha}{q_1}}
\frac{s^{\frac{\alpha}{q_1}}\theta_s}{s}
\\&\leq \gamma\left[\left(\ \fiint_{Q_{3s}(z_0)}|Du^m|^{2q_1}\,\mathrm{d}x\mathrm{d}t\right)^{\frac{1}{q_1}}+M_0^2\right]
\\&\quad+\gamma \frac{1}{s^{1-m}\theta_s^{1-m}}
\left(\ \fiint_{Q_{3s}}\psi^{1-m}\,\mathrm {d}x\mathrm{d}t\right)^2
M_0^{2m}
 +\gamma \frac{1}{s^{m-1}\theta_s^{1+m}}M_0^2
\\&=:L_1+L_2+L_3,
\end{split}\end{equation*}
with the obvious meaning of $L_1$, $L_2$ and $L_3$. We first consider the estimate for $L_2$.
Since $u\geq \psi$, we apply Lemma \ref{subcylinder} (1), (4) and H\"older's inequality to deduce
\begin{equation*}\begin{split}
\fiint_{Q_{3s}}\psi^{1-m}\,\mathrm {d}x\mathrm{d}t\leq\ \fiint_{Q_{3s}}u^{1-m}\,\mathrm {d}x\mathrm{d}t
\leq \left(\ \fiint_{Q_{3s}}u^{m+1}\,\mathrm {d}x\mathrm{d}t\right)^{\frac{1-m}{m+1}}\leq \theta_{3s}\leq c\theta_s.
\end{split}\end{equation*}
This implies that
\begin{equation*}\begin{split}L_2&\leq \gamma\frac{\theta_s^{m+1}}{s^{1-m}}M_0^{2m}=\left(\frac{\theta_s^{\frac{m+1}{1-m}}}{s}\right)^{1-m}M_0^{2m}
\leq \frac{1}{2}\frac{\theta_s^{\frac{m+1}{1-m}}}{s}+cM_0^2.
\end{split}\end{equation*}
Next, we rewrite $L_3$ as follows:
\begin{equation*}\begin{split}L_3=\gamma \frac{1}{\left(\frac{\theta_s^{\frac{m+1}{1-m}}}{s}\right)^{1-m}}M_0^2.
\end{split}\end{equation*}
Combining the estimates above, we arrive at
\begin{equation*}\begin{split}\frac{\theta_s^{\frac{m+1}{1-m}}}{s}&\leq \gamma\left[\left(\ \fiint_{Q_{3s}(z_0)}|Du^m|^{2q_1}\,\mathrm{d}x\mathrm{d}t\right)^{\frac{1}{q_1}}+M_0^2\right]+\frac{1}{2}\frac{\theta_s^{\frac{m+1}{1-m}}}{s}+
\gamma \frac{1}{\left(\frac{\theta_s^{\frac{m+1}{1-m}}}{s}\right)^{1-m}}M_0^2.
\end{split}\end{equation*}
Observe that we can reabsorb the second
term $\frac{1}{2}\frac{\theta_s^{\frac{m+1}{1-m}}}{s}$ on the right-hand side into the left.
It follows that
\begin{equation}\begin{split}\label{theta/s}\frac{\theta_s^{\frac{m+1}{1-m}}}{s}&\leq \gamma\left[\left(\ \fiint_{Q_{3s}(z_0)}|Du^m|^{2q_1}\,\mathrm{d}x\mathrm{d}t\right)^{\frac{1}{q_1}}+M_0^2\right]+
\gamma \frac{1}{\left(\frac{\theta_s^{\frac{m+1}{1-m}}}{s}\right)^{1-m}}M_0^2.
\end{split}\end{equation}
At this point, we claim that
\begin{equation}\begin{split}\label{claimedtheta/s}\frac{\theta_s^{\frac{m+1}{1-m}}}{s}&\leq \gamma\left(\ \fiint_{Q_{3s}(z_0)}|Du^m|^{2q_1}\,\mathrm{d}x\mathrm{d}t\right)^{\frac{1}{q_1}}+\gamma M_0^2+
1.
\end{split}\end{equation}
In the case $s^{-1}\theta_s^{\frac{m+1}{1-m}}\leq 1$, it is easy to see that \eqref{claimedtheta/s} holds trivially. In the case
$s^{-1}\theta_s^{\frac{m+1}{1-m}}> 1$, the desired estimate \eqref{claimedtheta/s} directly follows from \eqref{theta/s}.
This proves \eqref{claimedtheta/s} and therefore the proof of Proposition \ref{degenerate regime}
is complete.
\end{proof}
\subsection{The non-degenerate alternative}
In this subsection, we prove the reverse H\"older inequality analogue to \eqref{degenerate result} for the non-degenerate regime.
The treatment for non-degenerate case is different from the degenerate case.
\begin{proposition}\label{non degenerate regime}
Let $u$ be a nonnegative weak solution to the obstacle problem in the sense of Definition \ref{definition}.
Fix a point $z_0\in \Omega_T$ and suppose that $Q_{R,R^2}(z_0)\subset\Omega_T$.
Let $0<s\leq R^2$ and
suppose that the cylinder $Q_{s}(z_0)$ is intrinsic and satisfies \eqref{non degenerate}. Moreover, assume that $\psi^m$ is locally Lipschitz
continuous and
\begin{equation*}\begin{split}
\sup_{Q_{s}(z_0)}(\ \Psi^{\frac{1}{m+1}}+|\partial_t\psi^{1-m}|^{\frac{1}{1-m}}\ )\leq M_0,
\end{split}\end{equation*}
for some $M_0>0$. Then there exists $q_1\in(\frac{1}{2},1)$, depending only upon $n$ and $m$, such that the following holds:
\begin{equation}\begin{split}\label{non degenerate result} \ \fiint_{Q_{\frac{s}{2}}(z_0)}|Du^m|^2
\,\mathrm{d}x\mathrm{d}t \leq c_\epsilon
\left(\ \fiint_{Q_{s}(z_0)}|Du^m|^{2q_1}\,\mathrm{d}x\mathrm{d}t\right)^{\frac{1}{q_1}}
+c_\epsilon M_0^2+1.
\end{split}\end{equation}
\end{proposition}
\begin{proof}For simplicity of presentation, we assume that $z_0=(0,0)$. Let us first construct a smooth function $\eta\in C_0^\infty(B_{r(s)})$ satisfying
$0\leq\eta\leq1$ in $B_{r(s)}$, $\eta\equiv 1$ in $B_{r(s/2)}$ and $|D\eta|\leq 2(r(s)-r(s/2))^{-1}$. Define
\begin{equation*}\begin{split}\lambda=\ \dashint_{\Lambda_{s}}(u(t))_{B_{r(s)}}^\eta\,
\mathrm{d}t\qquad\text{and}\qquad \lambda(t)=(u(t))_{B_{r(s)}}^\eta,
\end{split}\end{equation*}
where $t\in\Lambda_s$.
Let $\sigma_1$, $\sigma_2\in [\frac{1}{2},1]$ and $\sigma_1<\sigma_2$.
Applying the similar argument as in the proof of \cite[Proposition 6.3]{GS}, we infer from \eqref{non degenerate} that
\begin{equation}\begin{split}\label{thetalambda}c_1\theta_s^{\frac{1}{1-m}}\leq \lambda\leq c_2\theta_s^{\frac{1}{1-m}}.\end{split}\end{equation}
Furthermore, we apply the Caccioppoli estimate \eqref{Cac2} with $(c,\,Q_{r_1,s_1},\,Q_{r_2,s_2})$
replaced by $(\lambda,\,Q_{\sigma_1s},\,Q_{\sigma_2s})$ to obtain
\begin{equation}\begin{split}\label{nondegenerate Cac}\frac{1}{s}\esssup_{t\in \Lambda_{\sigma_1s}}&\ \dashint_{B_{r(\sigma_1s)}}|u-\lambda|\,|u^m-\lambda^m|\,
\mathrm{d}x+\ \fiint_{Q_{\sigma_1 s}}|Du^m|^2\,\mathrm{d}x\mathrm{d}t
\\&\leq \frac{\gamma}{(\sigma_2-\sigma_1)s}\ \fiint_{Q_{\sigma_2s}}|u-\lambda|\,|u^m-\lambda^m|\,
\mathrm{d}x\mathrm{d}t\\&\quad+\frac{\gamma}{(\sigma_2-\sigma_1)^2r(s)^2}\ \fiint_{Q_{\sigma_2s}}|u^m-\lambda^m|^2\,\mathrm{d}x\mathrm{d}t
+\gamma\ \fiint_{Q_s}\Psi\,
\mathrm{d}x\mathrm{d}t,\end{split}\end{equation}
since $r(\sigma_2s)-r(\sigma_1s)\geq c(\sigma_2-\sigma_1)^{\hat b}r(s/2)\geq c(\sigma_2-\sigma_1)r(s)$. For any $\sigma\in [\frac{1}{2},1]$,
we set
\begin{equation*}\begin{split}T_1(\sigma)=\frac{1}{s}\esssup_{t\in \Lambda_{\sigma s}}\ \dashint_{B_{r(\sigma s)}}|u-\lambda|\,|u^m-\lambda^m|\,
\mathrm{d}x\quad\text{and}\quad T_2(\sigma)=\ \fiint_{Q_{\sigma s}}|Du^m|^2\,\mathrm{d}x\mathrm{d}t.
\end{split}\end{equation*}
We now choose $q_1\in(\frac{1}{2},1)$ satisfying \eqref{q1}. According to the proof of \cite[Proposition 6.3]{GS}, we infer from
\eqref{nondegenerate Cac} that
\begin{equation}\begin{split}\label{T1T2}T_1(\sigma_1)+T_2(\sigma_1)
&\leq\frac{1}{2}T_1(\sigma_2)
+\gamma\ \fiint_{Q_s}\Psi\,\mathrm{d}x\mathrm{d}t
\\&\quad+\gamma \frac{1}{(\sigma_2-\sigma_1)^{\frac{2}{\alpha}}}\left(\ \fiint_{Q_s}|Du^m|^{2q_1}\,\mathrm{d}x\mathrm{d}t\right)^{\frac{1}{q_1}}
\\&\qquad+\gamma \frac{1}{(\sigma_2-\sigma_1)^{\frac{2}{\alpha}}r(s)^2}\ \dashint_{\Lambda_s}|\lambda^m-\lambda(t)^m|^2\,\mathrm{d}t,
\end{split}\end{equation}
where $\alpha=\frac{2q_1m}{m+1}$. It remains to
treat the third term on the right-hand side of \eqref{T1T2}.
We first consider the case $r(s)^{\frac{1}{m}}> \frac{s}{r(s)}$.
Since $\lambda\leq c_2\theta_s^{\frac{1}{1-m}}$, we deduce
\begin{equation*}\begin{split}r(s)^{\frac{1}{m}-1}>\theta_s\geq c_2^{m-1}\lambda^{1-m}\end{split}\end{equation*}
and therefore $r>c_2^{-m}\lambda^m$. Recalling that $\lambda(t)=(u(t))_{B_{r(s)}}^\eta$, we obtain $\lambda(t)^{m+1}\leq c(u(t)^{m+1})_{B_{r(s)}}$.
Consequently, we deduce that
\begin{equation*}\begin{split}
\frac{1}{r(s)^2}&\ \dashint_{\Lambda_s}|\lambda^m-\lambda(t)^m|^2\,\mathrm{d}t\leq
2\frac{\lambda^{2m}}{r(s)^2}+\frac{2}{r(s)^2}\ \dashint_{\Lambda_s}\lambda(t)^{2m}\,\mathrm{d}t
\\&\leq c+c\frac{1}{\lambda^{2m}}\left(\ \dashint_{\Lambda_s}\lambda(t)^{m+1}\,\mathrm{d}t\right)^{\frac{2m}{m+1}}
\leq c+c\frac{1}{\lambda^{2m}}\left(\ \fiint_{Q_s}u^{m+1}\,\mathrm{d}t\right)^{\frac{2m}{m+1}}
\leq c+c\frac{\theta_s^{\frac{2m}{1-m}}}{\lambda^{2m}}\leq c.
\end{split}\end{equation*}
Next, we turn our attention to the case $r(s)^{\frac{1}{m}}\leq \frac{s}{r(s)}$.
We apply the inequality \eqref{gluing 4} from Corollary \ref{gluing corollary} to obtain
\begin{equation*}\begin{split}
\frac{1}{r(s)^2}\ \dashint_{\Lambda_s}&|\lambda^m-\lambda(t)^m|^2\,\mathrm{d}t
\leq\frac{\lambda^{2(m-1)}}{r(s)^2}\ \dashint_{\Lambda_s}|\lambda-\lambda(t)|^2\,\mathrm{d}t
\\&=\frac{\lambda^{2(m-1)}}{r(s)^2}
\ \dashint_{\Lambda_{s}}\dashint_{\Lambda_{s}}\big|(u(t))_{B_{r(s)}}^\eta-(u(t^\prime))_{B_{r(s)}}^\eta\big|^2\,
\mathrm{d}t\mathrm{d}t^\prime
\\&\leq \gamma \frac{\lambda^{2(m-1)}s^2}{r(s)^4}
\left(\ \fiint_{Q_s}|Du^m|\,\mathrm {d}x\mathrm{d}t+M_0\right)^2
\\&\quad+\gamma\lambda^{2(m-1)}\left(\ \fiint_{Q_s}\psi^{1-m}\,\mathrm {d}x\mathrm{d}t\right)^2M_0^{2m}
+\gamma \frac{\lambda^{2(m-1)}r(s)^{\frac{2}{m}}}{r(s)^{2}}M_0^2
\\&=:L_1+L_2+L_3,
\end{split}\end{equation*}
with the obvious meaning of $L_1$, $L_2$ and $L_3$. We first consider the estimate for $L_1$. From \eqref{thetalambda}, we get
\begin{equation*}\begin{split}L_1\leq \gamma \frac{s^2}{\theta_s^2r(s)^4}
\left(\ \fiint_{Q_{r_2,s}(z_0)}|Du^m|\,\mathrm {d}x\mathrm{d}t+M_0\right)^2\leq \gamma
\left(\ \fiint_{Q_{r_2,s}(z_0)}|Du^m|\,\mathrm {d}x\mathrm{d}t+M_0\right)^2.
\end{split}\end{equation*}
To estimate $L_2$, we recall that $u\geq \psi$. From \eqref{thetalambda} and H\"older's inequality, we obtain
\begin{equation*}\begin{split}L_2&\leq \gamma\theta_s^{-2}\left(\ \fiint_{Q_s}\psi^{m+1}\,\mathrm {d}x\mathrm{d}t\right)^{\frac{2(1-m)}
{m+1}}M_0^{2m}\leq \gamma\theta_s^{-2}\left(\ \fiint_{Q_s}u^{m+1}\,\mathrm {d}x\mathrm{d}t\right)^{\frac{2(1-m)}
{m+1}}M_0^{2m}
\\&\leq \gamma\theta_s^{-2}\theta_s^2M_0^{2m}\leq \gamma M_0^2+1.
\end{split}\end{equation*}
Finally, we come to the estimate of $L_3$. Recalling that $r(s)^{\frac{1}{m}}\leq \frac{s}{r(s)}$, we have
\begin{equation*}\begin{split}L_3\leq \gamma \frac{s^2}{\theta_s^2r(s)^4}M_0^2\leq \gamma M_0^2,
\end{split}\end{equation*}
since $\lambda\geq c_1\theta_s^{\frac{1}{1-m}}$. Consequently, we arrive at
\begin{equation*}\begin{split}T_1(\sigma_1)+T_2(\sigma_1)&
\leq\frac{1}{2}T_1(\sigma_2)
+\gamma\ \fiint_{Q_s}\Psi\,\mathrm{d}x\mathrm{d}t
\\&+\gamma \frac{1}{(\sigma_2-\sigma_1)^{\frac{2}{\alpha}}}\left[\left(\ \fiint_{Q_s}|Du^m|^{2q_1}\,\mathrm{d}x\mathrm{d}t\right)^{\frac{1}{q_1}}
+M_0^2+1
\right],
\end{split}\end{equation*}
for any $\sigma_1$, $\sigma_2\in [\frac{1}{2},1]$ and $\sigma_1<\sigma_2$. Now, we apply the iteration result from \cite[Lemma 2.1]{BDKS}
to reabsorb
the first term on the right-hand side into the left. We have thus proved the proposition.
\end{proof}
\section{Proof of the main result}
This section is devoted to the proof of Theorem \ref{main theorem}.
Our proof uses a certain stopping time argument which was introduced by Gianazza and Schwarzacher
\cite{GS}. However, in the context of obstacle problem, the argument is considerably more delicate.
We first point out that the scaling argument does not seem to work for the obstacle problem.
On the other hand, in order to use the reverse H\"older inequalities in \S 6, we introduce a certain
localized-centered maximal function. We adopt this kind of maximal function
to construct the superlevel sets of the gradient.

We now turn to the proof of Theorem \ref{main theorem}. Recalling that we have assumed $\mathfrak z_0=(0,0)$.
In this case, the assumption \eqref{lipschitz} reads
\begin{equation*}\begin{split}
\sup_{Q_{8R,64R^2}}(\ \Psi^{\frac{1}{m+1}}+|\partial_t\psi^{1-m}|^{\frac{1}{1-m}}\ )\leq M_0.
\end{split}\end{equation*}
For simplicity of presentation, we abbreviate $Q_{2R,4R^2}$ to $\hat Q$.
For $L\gg1$ to be fixed later, we introduce a localized average function
\begin{equation*}\begin{split}T_s(f)(z)=\ \fiint_{Q_s(z)}f\chi_{\hat Q}\,\mathrm{d}x\mathrm{d}t,
\quad 0<s<L^{-1}R^2,\end{split}\end{equation*}
where $f$ is a locally integrable function in $\Omega_T$.
Moreover, for any $f\in L^1_{\mathrm{loc}}(\Omega_T)$, we define the localized-centered maximal function as follows:
\begin{equation*}\begin{split}T^*(f)(z)=\sup_{0<s<L^{-1}R^2}
T_s(f)(z).\end{split}\end{equation*}
Next, we remark that this kind of one-parameter maximal function is different from the Hardy-Littlewood maximal function and it is of interest to
know whether $f(z)\leq T^*(f)(z)$ for almost every $z\in \hat Q$.
This motivates us to establish the following Lemma.
\begin{lemma}\label{maxlemma} For any $f\geq0$ with $f\in L_{\mathrm{loc}}^1(\Omega_T)$, there holds
\begin{equation}\begin{split}\label{maximal}f(z)\leq T^*(f)(z)\end{split}\end{equation}
for almost every $z\in \hat Q$.
\end{lemma}
\begin{proof} Our first goal is to establish the following weak type estimate
\begin{equation}\begin{split}\label{weak type estimate}\big|\big\{z\in\hat Q:\ T^*(f)(z)>\lambda\big\}\big|\leq c^\prime\frac{\|f\|_{L^1(\hat Q)}}{\lambda},\end{split}\end{equation}
where the constant $c^\prime$ depends only on $n$ and $m$.
In order to prove \eqref{weak type estimate}, we note that for any $z\in \hat Q\cap \big\{T^*(f)(z)>\lambda\big\}$ there exists
$s_z\leq L^{-1}R^2$ such that
\begin{equation*}\begin{split}\fiint_{Q_{s_z}(z)}f\chi_{\hat Q}\,\mathrm{d}x\mathrm{d}t>\lambda.\end{split}\end{equation*}
Moreover, the collection $\{Q_{s_z}(z): z\in \hat Q\cap \big\{T^*(f)(z)>\lambda\big\}\}$ forms a covering of the set
$\hat Q\cap \big\{T^*(f)(z)>\lambda\big\}$. From Lemma \ref{vitali}, we find a countable subfamily $\{Q_{s_{z_i}}(z_i)\}_{i=1}^\infty$
of pairwise disjoint cylinders, such that
\begin{equation*}\begin{split}\hat Q\cap \big\{T^*(f)(z)>\lambda\big\}\subset
\bigcup_{i=1}^\infty Q_{\chi s_{z_i}}(z_i),\end{split}\end{equation*}
where $\chi=\chi(n,m)$ is the constant defined in Lemma \ref{vitali}. By Lemma \ref{subcylinder} (4), there exists a constant $c^\prime=c^\prime(n,m)>1$ such that
\begin{equation*}\begin{split}
\big|\big\{z\in\hat Q:\ T^*(f)(z)>\lambda\big\}\big|&\leq \sum_{i=1}^\infty|Q_{\chi s_{z_i}}(z_i)|\leq c^\prime
\sum_{i=1}^\infty|Q_{s_{z_i}}(z_i)|
\\&\leq \frac{c^\prime}{\lambda} \sum_{i=1}^\infty \iint_{Q_{s_{z_i}}(z_i)}f\chi_{\hat Q}\,\mathrm{d}x\mathrm{d}t
\leq c^\prime\frac{\|f\|_{L^1(\hat Q)}}{\lambda},
\end{split}\end{equation*}
which gives \eqref{weak type estimate}. Next, we define a sequence of localized average functions
\begin{equation*}\begin{split}T_j(f)(z)=\ \fiint_{Q_{s_j}(z)}f\chi_{\hat Q}\,\mathrm{d}x\mathrm{d}t,\qquad j=1,2,\cdots,\end{split}\end{equation*}
where $s_j=2^{-j}L^{-1}R^2$. Another step in the proof is to show that the sequence $T_j(f)$ converges in measure on $\hat Q$ to $f$.
To this end, we fix $\epsilon_0>0$. For any $\epsilon>0$, we choose $g\in C_0(\hat Q)$ such that $\|g -f\chi_{\hat Q}\|_{L^1(\Omega_T)}<\frac{1}{3c^\prime}\epsilon \epsilon_0$. Furthermore, there exists an integer $N_0=N_0(\epsilon_0,\epsilon,R)$ such that for
any $j\geq N_0$ there holds
$$|g(z)-g(z^\prime)|\leq
\frac{1}{3}\epsilon_0\qquad\text{for}\ \ \text{all}\qquad z^\prime\in Q_{s_j^{\hat b}R^{1-2\hat b},s_j}(z).$$
From Lemma \ref{subcylinder} (2), we see that $Q_{s_j}(z)\subset Q_{s_j^{\hat b}R^{1-2\hat b},s_j}(z)$
and consequently
\begin{equation*}\begin{split}|T_j(g)(z)-g(z)|\leq
\frac{1}{3}\epsilon_0\end{split}\end{equation*}
holds for all $j\geq N_0$. It follows from \eqref{weak type estimate} that for any $j\geq N_0$ there holds
\begin{equation*}\begin{split}\big|\big\{z\in\hat Q:\ |T_j(f)(z)-f(z)|>\epsilon_0\big\}\big|
&\leq \big|\big\{z\in\hat Q:\ |T_j(f)(z)-T_j(g)(z)|>\frac{1}{3}\epsilon_0\big\}\big|
\\&\quad+\big|\big\{z\in\hat Q:\ |T_j(g)(z)-g(z)|>\frac{1}{3}\epsilon_0\big\}\big|
\\&\quad+\big|\big\{z\in\hat Q:\ |g(z)-f(z)|>\frac{1}{3}\epsilon_0\big\}\big|
\\&\leq \big|\big\{z\in\hat Q:\ |T^*(f-g)(z)|>\frac{1}{3}\epsilon_0\big\}\big|
+\frac{3}{\epsilon_0}\|f-g\|_{L^1(\hat Q)}
\\&\leq \epsilon.
\end{split}\end{equation*}
This proves that $T_j(f)$ converges in measure to $f$. Then there exists a subsequence $T_{j_k}(f)$ converging to $f$ almost everywhere.
It follows that for almost every $z\in \hat Q$, there holds
\begin{equation*}\begin{split}f(z)=\lim_{k\to\infty}T_{j_k}(f)(z)\leq T^*(f)(z),\end{split}\end{equation*}
which completes the proof.
\end{proof}
Furthermore, for $\sigma_1$, $\sigma_2\in [1,2]$ and $\sigma_1<\sigma_2$, we define two concentric cylinders $\hat Q_{\sigma_1}=Q_{\sigma_1R,\sigma_1^2R^2}$
and $\hat Q_{\sigma_2}=Q_{\sigma_2R,\sigma_2^2R^2}$.
We are interested in getting estimates on such concentric cylinders.
To this end, we first need the following lemma.
\begin{lemma}\label{L-1}Let $z\in \hat Q_{\sigma_1}$ and $L_1\geq1$. If $L_1s\leq R^2$ and
$Q_{L_1s}(z)\cap (\hat Q_{\sigma_2})^c\neq \emptyset$, then
\begin{equation*}\begin{split}\fiint_{Q_{s}(z)}|Du^m|^2\chi_{\hat Q}\,\mathrm{d}x\mathrm{d}t
\leq \mu_{\sigma_1,\sigma_2,L_1},\end{split}\end{equation*}
where
\begin{equation*}\begin{split}
\mu_{\sigma_1,\sigma_2,L_1}=\frac{\gamma}{(\sigma_2-\sigma_1)^\gamma }\left(\frac{1}{R^2}+\ \fiint_{Q_{4R,16R^2}}\Psi\,\mathrm{d}x\mathrm{d}t
\right)
\end{split}\end{equation*}
and the constant $\gamma$ depends only on $n$, $m$, $\nu_0$, $\nu_1$ and $L_1$.
\end{lemma}
The proof of Lemma \ref{L-1} is quite similar to \cite[Lemma 7.1]{GS} and so is omitted.
The crucial result in our proof of Theorem \ref{main theorem} will be the following proposition, which is analogue to \cite[Proposition 7.2]{GS}.
\begin{proposition}\label{reverse holder} Fix a point $z\in\hat Q_{\sigma_1}\cap \big\{T^*(|Du^m|^2)(z)>\lambda\big\}$. Suppose that $\lambda\geq \lambda_0$, where
\begin{equation*}\begin{split}\lambda_0=\mu_{\sigma_1,\sigma_2,L_1}+\ \fiint_{Q_{4R,16R^2}}|Du^m|^2\,\mathrm{d}x\mathrm{d}t
\end{split}\end{equation*}
and $L_1=7\chi$.
Then there exist $q_1\in(\frac{1}{2},1)$, $L=L(n,m,\nu_0,\nu_1)>10\chi$ and $s_z\in (0,2L^{-1}R^2]$ such that the following holds:
\begin{itemize}
  \item [(1)] There exists a constant $c_1=c_1(n,m,\nu_0,\nu_1)$ such that
  \begin{equation}\label{reverse holder final}\lambda\leq  c_1\ \fiint_{Q_{s_z}(z)}|Du^m|^2
\,\mathrm{d}x\mathrm{d}t \leq c_1
\left(\ \fiint_{Q_{3s_z}(z)}|Du^m|^{2q_1}\,\mathrm{d}x\mathrm{d}t\right)^{\frac{1}{q_1}}
+c_1 M_0^2+c_1R^{-2}+1.\end{equation}
  \item [(2)] We have $3\chi s_z\leq R^2$, $Q_{3\chi s_z}(z)\subset \hat Q_{\sigma_2}$ and
  \begin{equation}\label{reverse inequality}\fiint_{Q_{3\chi s_z}(z)}|Du^m|^2
\,\mathrm{d}x\mathrm{d}t \leq c_2\lambda,\end{equation}
where the constant $c_2$ depends only on $n$, $m$, $\nu_0$ and $\nu_1$.
\end{itemize}
\end{proposition}
\begin{proof} The proof is adapted from \cite[Proposition 7.2]{GS}. For any fixed point $z\in\hat Q_{\sigma_1}\cap \big\{T^*(|Du^m|^2)(z)>\lambda\big\}$, we set
\begin{equation*}\begin{split}
s_z^\prime=\sup\left\{s:\ \ \fiint_{Q_s(z)}|Du^m|^2\chi_{\hat Q}\,\mathrm{d}x\mathrm{d}t>\lambda,
\quad 0<s\leq L^{-1}R^2\right\}.
\end{split}\end{equation*}
It follows that
\begin{equation*}\begin{split}\fiint_{Q_{s_z^\prime}(z)}|Du^m|^2\chi_{\hat Q}\,\mathrm{d}x\mathrm{d}t\geq \lambda.
\end{split}\end{equation*}
Therefore, by $\lambda\geq \lambda_0>\mu_{\sigma_1,\sigma_2,L_1}$ and Lemma \ref{L-1}, we have $Q_{L_1s_z^\prime}(z)\subset \hat Q_{\sigma_2}$.
Let $L_2> 4$ be a constant which will be chosen later, and assume that $L\geq 5L_2+10\chi$. This implies that $(2s_z^\prime, L_2s_z^\prime)\subset (0,R^2]$.

In order to prove \eqref{reverse holder final}, we first consider the case that there exists $s\in (2s_z^\prime, L_2s_z^\prime)$ such that
\begin{equation*}\begin{split}
\fiint_{Q_s(z)}u^{m+1}\,\mathrm{d}x\mathrm{d}t= \theta_s(z)^{\frac{m+1}{1-m}}.
\end{split}\end{equation*}
By Lemma \ref{subcylinder} (5), we infer that $Q_{2s_z^\prime}(z)$ is intrinsic. If the cylinder $Q_{2s_z^\prime}(z)$
is degenerate, then
we choose $s_z=2s_z^\prime$ and \eqref{reverse holder final} follows from Proposition \ref{degenerate regime}. On the other hand, if the cylinder $Q_{2s_z^\prime}(z)$ is non-degenerate, then we choose $s_z=s_z^\prime$ and \eqref{reverse holder final} follows from Proposition \ref{non degenerate regime}.
Next, we turn our attention to the case that the strict inequality
\begin{equation*}\begin{split}
\fiint_{Q_s(z)}u^{m+1}\,\mathrm{d}x\mathrm{d}t< \theta_s(z)^{\frac{m+1}{1-m}}
\end{split}\end{equation*}
holds for any $s\in (2s_z^\prime, L_2s_z^\prime)$. In this case, we set $s_z=s_z^\prime$ and define
\begin{equation*}\begin{split}
\sigma_z=\inf \left\{s:\ \fiint_{Q_s(z)}u^{m+1}\,\mathrm{d}x\mathrm{d}t=\theta_s(z)^{\frac{m+1}{1-m}},\ 2s_z<s\leq R^2\right\}.
\end{split}\end{equation*}
By hypothesis, we see that $\sigma_z\in [L_2s_z, R^2]$ and $Q_{\sigma_z}(z)$ is a sub-intrinsic cylinder constructed in \S 3.
Since $s_z\leq L_2^{-1}\sigma_z$, we use Lemma \ref{subcylinder} (3) to obtain a decay estimate
\begin{equation}\begin{split}\label{thetadecay}
\frac{\theta_{2s_z}(z)^{\frac{m+1}{1-m}}}{2s_z}\leq \left(\frac{2}{L_2}\right)^{\frac{m+1}{1-m}\beta-1}\frac{\theta_{\sigma_z}(z)^{\frac{m+1}{1-m}}}{\sigma_z},
\end{split}\end{equation}
where $\beta=1-2\hat b$ and $\frac{m+1}{1-m}\beta-1>0$.
Next, we invoke Lemma \ref{Caccioppoli}, takes the form
\begin{equation}\begin{split}\label{energy}
\lambda<\fiint_{Q_{s_z}(z)}|Du^m|^2\,\mathrm{d}x\mathrm{d}t\leq \gamma \frac{\theta_{2s_z}(z)^{\frac{m+1}{1-m}}}{2s_z}+\gamma M_0^{m+1}.
\end{split}\end{equation}
In the case $\sigma_z\in [\frac{1}{12 L}R^2, R^2]$, we apply Lemma \ref{subcylinder} (4), (6) to obtain
\begin{equation*}\begin{split}
\frac{\theta_{\sigma_z}(z)^{\frac{m+1}{1-m}}}{\sigma_z}\leq \frac{12L}{R^2}\left(\frac{R^2}{\sigma_z}\right)^{2\hat a\frac{m+1}{1-m}}
\theta_{R^2}(z)^{\frac{m+1}{1-m}}\leq c\frac{1}{R^2}.
\end{split}\end{equation*}
Combining this with \eqref{thetadecay} and \eqref{energy} we find that
\begin{equation*}\begin{split}
\lambda<\fiint_{Q_{s_z}(z)}|Du^m|^2\,\mathrm{d}x\mathrm{d}t\leq c\frac{1}{R^2}+\gamma M_0^{m+1},
\end{split}\end{equation*}
which proves the estimate \eqref{reverse holder final}. Furthermore, we consider the case $\sigma_z\in (L_2s_z,\frac{1}{12 L}R^2)$.
In this case, the cylinder $Q_{\sigma_z}(z)$ is intrinsic and satisfies
\begin{equation*}\fiint_{Q_{\sigma_z}(z)}u^{m+1}\,\mathrm{d}x\mathrm{d}t=\theta_{\sigma_s}(z)^{\frac{m+1}{1-m}}.\end{equation*}
It follows that
\begin{equation*}\begin{split}
\fiint_{Q_{\frac{\sigma_z}{2}}(z)}u^{m+1}\,\mathrm{d}x\mathrm{d}t<\theta_{\frac{\sigma_z}{2}}(z)^{\frac{m+1}{1-m}}
\leq \left(\frac{1}{2}\right)^{\beta\frac{m+1}{1-m}}\theta_{\sigma_z}(z)^{\frac{m+1}{1-m}}
=\left(\frac{1}{2}\right)^{\beta\frac{m+1}{1-m}}
\fiint_{Q_{\sigma_z}(z)}u^{m+1}\,\mathrm{d}x\mathrm{d}t.
\end{split}\end{equation*}
From \cite[Lemma 2.4]{GS}, we find that the cylinder $Q_{\sigma_z}(z)$ is degenerate. This enables us to use \eqref{claimedtheta/s} from the proof of Proposition
\ref{degenerate}. Then, there exists a constant $\hat \gamma=\hat\gamma(n,m,\nu_0,\nu_1)$ such that
\begin{equation*}\begin{split}
\frac{\theta_{\sigma_z}(z)^{\frac{m+1}{1-m}}}{\sigma_z}&\leq \hat\gamma\left(\ \fiint_{Q_{3\sigma_z}(z)}|Du^m|^{2q_1}\,\mathrm{d}x\mathrm{d}t\right)^{\frac{1}{q_1}}+\hat \gamma M_0^2+
1
\leq \hat\gamma \lambda+\hat\gamma M_0^2+
1,
\end{split}\end{equation*}
since $s_z^\prime<3\sigma_z<L^{-1}R^2$.
Combining this with \eqref{thetadecay} and \eqref{energy} we finally arrive at
\begin{equation}\begin{split}\label{lambdalambda}
\lambda<&\fiint_{Q_{s_z}(z)}|Du^m|^2\,\mathrm{d}x\mathrm{d}t\leq  \gamma \left(\frac{2}{L_2}\right)^{\frac{m+1}{1-m}\beta-1}\frac{\theta_{\sigma_z}(z)^{\frac{m+1}{1-m}}}{\sigma_z}
+\gamma M_0^{m+1}
\\&\leq  \hat \gamma \gamma \left(\frac{2}{L_2}\right)^{\frac{m+1}{1-m}\beta-1}\lambda+cM_0^2+1.
\end{split}\end{equation}
In \eqref{lambdalambda} we choose
$L_2=2(2\gamma\hat\gamma)^{\frac{1}{\frac{m+1}{1-m}\beta-1}}$
and this determines the constant
\begin{equation*}\begin{split}L=5L_2+10\chi=10(2\gamma\hat\gamma)^{\frac{1}{\frac{m+1}{1-m}\beta-1}}+10\chi.
\end{split}\end{equation*}
Therefore, we can reabsorb the first term on the right-hand side of \eqref{lambdalambda} into the left and this proves the
estimate \eqref{reverse holder final}. 

On the other hand, for such a choice of $L$, we see immediately that 
\begin{equation*}3\chi s_z\leq 6\chi s_z^{\prime}
\leq 6\chi L^{-1}R^2\leq R^2\quad\text{and}\quad
Q_{3\chi s_z}(z)\subset Q_{6\chi s_z^\prime}(z)\subset Q_{L_1 s_z^\prime}(z)\subset\hat Q_{\sigma_2}.\end{equation*}
Finally, we come to the proof of
\eqref{reverse inequality}. To this end, we have to distinguish two cases, whether $3\chi s_z<L^{-1}R^2$, or $3\chi s_z\geq L^{-1}R^2$. In the case $3\chi s_z<L^{-1}R^2$, the inequality \eqref{reverse inequality} follows directly from the definition of $s_z^\prime$. Next, we consider the second case. From Lemma \ref{subcylinder} (4), (6), we obtain
\begin{equation*}\begin{split}
\theta_{3\chi s_z}\leq \left(\frac{R^2}{3\chi s_z}\right)^{2\hat a}  \theta_{R^2}\leq cL^{2\hat a}
\quad\text{and}\quad
\theta_{3\chi s_z}\geq \left(\frac{3\chi s_z}{R^2}\right)^\beta \theta_{R^2}\geq L^{-\beta}.
\end{split}\end{equation*}
This implies that
\begin{equation*}\begin{split}
r(3\chi s_z)=\sqrt{\frac{3\chi s_z}{\theta_{3\chi s_z}}}\geq c^{-\frac{1}{2}}L^{-\hat a-\frac{1}{2}}R
\end{split}\end{equation*}
and consequently
\begin{equation*}\begin{split}
\fiint_{Q_{3\chi s_z}(z)}|Du^m|^2\,\mathrm{d}x\mathrm{d}t\leq c\ \fiint_{Q_{4R,16R^2}}|Du^m|^2\,\mathrm{d}x\mathrm{d}t\leq c\lambda,
\end{split}\end{equation*}
which gives \eqref{reverse inequality}.
This finishes the proof of the proposition.
\end{proof}
With the help of Lemma \ref{maxlemma} and Proposition \ref{reverse holder},
we are now in a position to prove the main result.
The proof follows in a similar manner as the proof of \cite[Theorem 7.3]{GS} and we just sketch the proof.
\begin{proof}[Proof of Theorem \ref{main theorem}]
Let $\lambda\geq \lambda_0^\prime$, where
\begin{equation*}\begin{split}\lambda_0^\prime=\mu_{\sigma_1,\sigma_2,L_1}+\ \fiint_{Q_{4R,16R^2}}|Du^m|^2\,\mathrm{d}x\mathrm{d}t+16c_1 M_0^2+16c_1R^{-2}+16>\lambda_0\end{split}\end{equation*}
and $c_1$ be the constant in \eqref{reverse holder final} from Proposition \ref{reverse holder}.

For any fixed  $z\in\hat Q_{\sigma_1}\cap \big\{T^*(|Du^m|^2)(z)>\lambda\big\}$, let $s_z$ be the positive constant constructed in Proposition \ref{reverse holder}.
We note that the collection
\begin{equation*}\begin{split}\mathcal{F}=\left\{Q_{3s_z}(z):\ z\in \hat Q_{\sigma_1}\cap \big\{T^*(|Du^m|^2)(z)>\lambda\big\}\right\}\end{split}\end{equation*}
forms a covering of the superlevel set $\hat Q_{\sigma_1}\cap \big\{T^*(|Du^m|^2)(z)>\lambda\big\}$. From Lemma \ref{vitali}, there exists a countable subfamily $\{Q_{3s_{z_i}}(z_i)\}_{i=1}^\infty\subset \mathcal{F}$ of pairwise disjoint sub-intrinsic cylinders, such that
$\{Q_{3\chi s_{z_i}}(z_i)\}_{i=1}^\infty$ covers the superlevel set $\hat Q_{\sigma_1}\cap \big\{T^*(|Du^m|^2)(z)>\lambda\big\}$.
We abbreviate $Q_i=Q_{s_{z_i}}(z_i)$, $Q_i^*=Q_{3s_{z_i}}(z_i)$ and $Q_i^{**}=Q_{3\chi s_{z_i}}(z_i)$. Next, we choose
$\eta=(16c_1)^{-1}$.
Then, for each $i=1,2,\cdots$,
we infer from Proposition \ref{reverse holder} (1) that
\begin{equation*}\begin{split}\lambda^{q_1}\leq
\frac{c_1^{q_1}}{|Q_i^*|}\iint_{Q_i^*}|Du^m|^{2q_1}\chi_{\{|Du^m|^2>\eta \lambda\}}\,\mathrm{d}x\mathrm{d}t
+\frac{1}{2}\lambda^{q_1}\end{split}\end{equation*}
and hence
\begin{equation}\begin{split}\label{lambdaQi}\lambda|Q_i^*|\leq
2c_1^{q_1}\lambda^{1-q_1}\iint_{Q_i^*}|Du^m|^{2q_1}\chi_{\{|Du^m|^2>\eta \lambda\}}\,\mathrm{d}x\mathrm{d}t.
\end{split}\end{equation}
Recalling that $\{Q_i^{**}\}_{i=1}^\infty$ is a covering of the set $\hat Q_{\sigma_1}\cap \big\{T^*(|Du^m|^2)(z)>\lambda\big\}$,
we infer from Proposition \ref{reverse holder} (2), Lemma \ref{subcylinder} (4) and \eqref{lambdaQi} that
\begin{equation}\begin{split}\label{final estimate}
\iint_{\hat Q_{\sigma_1}\cap \big\{T^*(|Du^m|^2)(z)>\lambda\big\}}|Du^m|^2\,\mathrm{d}x\mathrm{d}t
\leq
\bar \gamma c_1^{q_1}\lambda^{1-q_1}\iint_{\hat Q_{\sigma_2}}|Du^m|^{2q_1}\chi_{\{|Du^m|^2>\eta \lambda\}}\,\mathrm{d}x\mathrm{d}t,
\end{split}\end{equation}
where the constant $\bar \gamma$ depends only upon $n$, $m$, $\nu_0$ and $\nu_1$.
Moreover, for some $\epsilon\in (0,1)$ to be specified later and $k>\lambda_0^\prime$,
we multiply both sides of \eqref{final estimate} by $\lambda^{-1+\epsilon}$ and integrate over the interval $(\lambda_0^\prime,k)$ with respect to $\lambda$.

To estimate a lower bound for the left-hand side of \eqref{final estimate},
we use the inequality \eqref{maximal} from Lemma \ref{maxlemma} to infer that
$\hat Q_{\sigma_1}\cap\big\{|Du^m|^2(z)>\lambda\big\}\subset \hat Q_{\sigma_1}\cap\big\{T^*(|Du^m|^2)(z)>\lambda\big\}$ and consequently
\begin{equation}\begin{split}\label{final1}
\int_{\lambda_0^\prime}^k&\lambda^{-1+\epsilon}\iint_{\hat Q_{\sigma_1}\cap \big\{T^*(|Du^m|^2)(z)>\lambda\big\}}|Du^m|^2\,\mathrm{d}x\mathrm{d}t
\mathrm{d}\lambda
\\&\geq
\int_{\lambda_0^\prime}^k\lambda^{-1+\epsilon}\iint_{\hat Q_{\sigma_1}\cap \big\{|Du^m|^2>\lambda\big\}}|Du^m|^2\,\mathrm{d}x\mathrm{d}t
\mathrm{d}\lambda
\\&\geq \int_0^k\lambda^{-1+\epsilon}\iint_{\hat Q_{\sigma_1}\cap \big\{|Du^m|^2>\lambda\big\}}|Du^m|^2\,\mathrm{d}x\mathrm{d}t
\mathrm{d}\lambda
-\frac{1}{\epsilon}(\lambda_0^\prime)^\epsilon\iint_{\hat Q_{\sigma_2}}|Du^m|^2\,\mathrm{d}x\mathrm{d}t
\\&=\frac{1}{\epsilon}\iint_{\hat Q_{\sigma_1}}|Du^m|^2\min\left\{k,|Du^m|^2\right\}^\epsilon\,\mathrm{d}x\mathrm{d}t
-\frac{1}{\epsilon}(\lambda_0^\prime)^\epsilon\iint_{\hat Q_{\sigma_2}}|Du^m|^2\,\mathrm{d}x\mathrm{d}t.
\end{split}\end{equation}
Next, we come to the estimate of the right-hand side of \eqref{final estimate}. We apply Fubini's theorem to obtain
\begin{equation}\begin{split}\label{final2}
\bar \gamma c_1^{q_1}&\int_{\lambda_0^\prime}^k\lambda^{\epsilon-q_1}\iint_{\hat Q_{\sigma_2}}|Du^m|^{2q_1}\chi_{\{|Du^m|^2>\eta \lambda\}}\,\mathrm{d}x\mathrm{d}t
\mathrm{d}\lambda
\\&\leq \frac{\bar \gamma c_1^{q_1}}{\epsilon-q_1+1}
\eta^{q_1-\epsilon-1}\iint_{\hat Q_{\sigma_2}}|Du^m|^{2q_1}\min\left\{k,|Du^m|^2\right\}^{\epsilon-q_1+1}\,\mathrm{d}x\mathrm{d}t
\\&\leq \frac{\bar \gamma c_1^{q_1}}{1-q_1}
\eta^{q_1-2}\iint_{\hat Q_{\sigma_2}}|Du^m|^2\min\left\{k,|Du^m|^2\right\}^\epsilon\,\mathrm{d}x\mathrm{d}t.
\end{split}\end{equation}
At this stage, we choose
$\epsilon=\frac{1-q_1}{2\bar \gamma c_1^{q_1}\eta^{q_1-2}}$.
Combining \eqref{final estimate}-\eqref{final2}, we obtain the following estimate
\begin{equation*}\begin{split}
\fiint_{\hat Q_{\sigma_1}}&|Du^m|^2\min\left\{k,|Du^m|^2\right\}^\epsilon\,\mathrm{d}x\mathrm{d}t
\\&\leq \frac{1}{2}\ \fiint_{\hat Q_{\sigma_2}}|Du^m|^2\min\left\{k,|Du^m|^2\right\}^\epsilon\,\mathrm{d}x\mathrm{d}t\\&\quad +
\gamma \left(\frac{1}{(\sigma_2-\sigma_1)^{-\gamma}}
(M_0^2+R^{-2}+1)\right)^\epsilon\ \fiint_{Q_{4R,16R^2}}|Du^m|^2\,\mathrm{d}x\mathrm{d}t
\\&\qquad +\gamma \left(\ \fiint_{Q_{4R,16R^2}}|Du^m|^2\,\mathrm{d}x\mathrm{d}t\right)^{1+\epsilon}.
\end{split}\end{equation*}
Finally, we use the iteration result from \cite[Lemma 2.1]{BDKS}
to reabsorb
the first term on the right-hand side into the left, and the desired estimate \eqref{main result estimate} is proved by letting $k\to\infty$.
This finishes the proof
of Theorem \ref{main theorem}.
\end{proof}
\bibliographystyle{abbrv}

\end{document}